\documentclass[12pt]{amsart}
\usepackage[T1]{fontenc}
\usepackage[latin1]{inputenc}
\usepackage{typearea}
\usepackage{geometry}
\usepackage{ulem}
\usepackage{amsmath}
\usepackage{amssymb}
\usepackage{latexsym}
\usepackage{enumerate}
\usepackage{amsthm}
\usepackage[all]{xy}
\usepackage{hhline}
\usepackage{epsf} 
\usepackage{cite}
\usepackage{verbatim}
\usepackage{dsfont}
\usepackage{mathrsfs}
\usepackage{color}

\newtheorem{theorem}{{\sc Theorem}}[section]
\newtheorem{cor}[theorem]{{\sc Corollary}}

\newtheorem{lemma}[theorem]{{\sc Lemma}}
\newtheorem{prop}[theorem]{{\sc Proposition}}

\theoremstyle{remark}
\newtheorem{remark}[theorem]{{\sc Remark}}

\theoremstyle{definition}

\newtheorem{example}[theorem]{\sc example}

\newcommand{\R}{\mathbb{R} }
\newcommand{\N}{\mathbb{N} }

\newcommand{\A}{\mathcal{A}}

\newcommand{\F}{\mathcal{F}}
\newcommand{\G}{\mathcal{G}}

\newcommand{\D}{\mathcal{D}}
\newcommand{\W}{\mathcal{W}}
\newcommand{\K}{\mathcal{K}}

\newcommand{\calS}{\mathcal{S}}

\newcommand{\Prob}{\mathbb{P}}
\newcommand{\E}{\mathbb{E}}

\newcommand{\Pot}{\mathcal{P}}

\providecommand{\abs}[1]{\lvert #1\rvert}
\providecommand{\babs}[1]{\bigl\lvert #1\bigr\rvert}
\providecommand{\Babs}[1]{\Bigl\lvert #1\Bigr\rvert}

\providecommand{\norm}[1]{\lVert #1\rVert}
\providecommand{\fnorm}[1]{\lVert #1\rVert_\infty}

\providecommand{\proj}[2]{\pro\bigl\{#1\,\bigl|\,C_{#2}\bigr\}}

\DeclareMathOperator{\Var}{Var}

\DeclareMathOperator{\dom}{dom}
\DeclareMathOperator{\Cov}{Cov}

\DeclareMathOperator{\Lip}{Lip}

\DeclareMathOperator{\Inf}{Inf}

\DeclareMathOperator{\pro}{proj}
\DeclareMathOperator{\Id}{Id}

\renewcommand{\phi}{\varphi}
\renewcommand{\epsilon}{\varepsilon}

\renewcommand{\rho}{\varrho}
\renewcommand{\P}{\mathbb{P}}

\providecommand{\absolute}[1]{\lvert #1\rvert}
\providecommand{\Bigabsolute}[1]{\Big\lvert #1\Big\rvert}
\renewcommand{\1}{\mathds{1}}

\usepackage{bm}
\newcommand{\ellnorm}[3]{\Vert{#3}\Vert_{\ell^{#1}(\N)^{\otimes{#2}}}}

\begin{document}
\title[Fourth moment condition]{On the fourth moment condition for Rademacher chaos}
\author{Christian D\"obler \and Kai Krokowski}
\thanks{\noindent Universit\'{e} du Luxembourg, Unit\'{e} de Recherche en Math\'{e}matiques, Ruhr-Universit\"at Bochum, Fakult\"at f\"ur Mathematik \\
E-mails: christian.doebler@uni.lu, kai.krokowski@rub.de }
\begin{abstract}  
Adapting the spectral viewpoint suggested in \cite{Led12} in the context of symmetric Markov diffusion generators and recently exploited in the non-diffusive setup of a Poisson random measure \cite{DP17}, 
we investigate the fourth moment condition for discrete multiple integrals with respect to general, i.e.\ non-symmetric and non-homogeneous, Rademacher sequences and show that, in this situation, the fourth moment 
alone does not govern the asymptotic normality. Indeed, here one also has to take into consideration the maximal influence of the corresponding kernel functions. In particular, we show that there is \textit{no} exact fourth moment 
theorem for discrete multiple integrals of order $m\geq2$ with respect to a symmetric Rademacher sequence. 
This behavior, which is in contrast to the Gaussian \cite{NP05} and Poisson \cite{DP17} situation, closely resembles the conditions for asymptotic normality of degenerate, non-symmetric $U$-statistics from the 
classical paper \cite{deJo90}.
\end{abstract}
\maketitle

\section{Introduction and main results}\label{intro}

\subsection{Motivation and outline}
The remarkable \textit{fourth moment theorem} \cite{NP05} by Nualart and Peccati states that a normalized sequence of multiple Wiener-It\^{o} integrals of fixed order on a Gaussian space converges in distribution to a standard normal random variable $N$, if and only if the corresponding sequence of fourth moments converges to $3$, i.e.\ to the fourth moment of $N$. 
The purpose of the present article is to discuss the validity of the fourth moment condition for sequences of discrete multiple integrals $(F_n)_{n \in \N}=(J_m(f_n))_{n \in \N}$ of order $m\in\N:=\{1,2,\dotsc\}$ of a general independent Rademacher sequence $X=(X_j)_{j\in\N}$, see below for precise definitions. 
As we will see, in contrast to the situation on a Gaussian space \cite{NP05} or on a Poisson space \cite{DP17}, in general, there is no \textit{exact} fourth moment theorem for Rademacher chaos. 
By this we mean that, in general, for a sequence $(F_n)_{n\in\N}$ of normalized discrete multiple integrals of a fixed order $m\in\N$ 
with respect to $X$, the convergence of $\E[F_n^4]$ to $3$ as $n\to\infty$ does not guarantee asymptotic normality of the sequence. However, the following positive result holds true: 
Whenever $\E[F_n^4]$ converges to $3$ and the \textit{maximal influence} $\sup_{k\in\N} \Inf_k(f_n)$ of the \textit{kernels} $f_n$ converges to $0$ as $n\to\infty$, then $F_n$ converges in distribution to $N$. 
Here, for a symmetric function $f:\N^m\rightarrow\R$ with 
\[\sum_{1\leq i_1<\ldots<i_m<\infty}f^2(i_1,\dotsc,i_m)<\infty\,,\]
the \textit{influence of the variable $k\in\N$ on $f$} is defined by
\begin{equation}\label{definf}
\Inf_k(f):=\sum_{\substack{(i_2,\dotsc,i_m)\in(\N\setminus\{k\})^{m-1}:\\ 1\leq i_2<\dotsc<i_m<\infty}}
f^2(k,i_2,\dotsc,i_m)\,.
\end{equation}
Interestingly, these \textit{influence functions} $k\mapsto \Inf_k(f)$ have raised a lot of attention recently. For instance, as demonstrated in the seminal papers \cite{MOO10} and \cite{NPR-aop}, they play a major role 
for the universality of multilinear polynomial forms with bounded degree. Furthermore, see again \cite{MOO10}, many recent problems and conjectures involving boolean functions with applications to theoretical computer science and social 
choice theory are only stated for \textit{low influence functions}, i.e.\ functions such that $\sup_{k\in\N}\Inf_k(f)$ is small.
The main reasons for this are that restricting oneself to low influence functions often excludes trivial and therefore non-relevant counterexamples, and, that these functions seem to be most interesting in applications. 

\subsection{Further historical comments and related results}\label{history}
In recent years, the fundamental result from \cite{NP05} has been amplified in many respects: On the one hand, it has been generalized to a multidimensional statement by Peccati and Tudor \cite{PecTu05} and, on the other hand, by combining Malliavin calculus and Stein's method of normal approximation, Nourdin and Peccati \cite{NP-ptrf} succeeded in providing error bounds on various probability  distances, including the total variation and Kolmogorov distances, between the law of a general smooth (in a Malliavin sense) functional on a Gaussian space and the standard normal distribution. In the special case of a multiple Wiener-It\^{o} integral their bounds can be expressed in terms 
of the fourth cumulant of the integral only.
We refer to the monograph \cite{NouPecbook} for a comprehensive treatment of results obtained by combining Malliavin calculus on a Gaussian space and Stein's method. 
This so-called \textit{Malliavin-Stein method} originating from the seminal paper \cite{NP-ptrf} is not restricted to a Gaussian framework, but roughly speaking, it may be set up whenever a version of Malliavin calculus is available for the respective probabilistic structure. To wit, shortly after the appearance of \cite{NP-ptrf}, in the papers \cite{PSTU} and \cite{NPR-ejp}, the respective groups of authors succeeded in combining Malliavin calculus on a general Poisson space and for functionals of a Rademacher sequence with Stein's method in order to obtain error bounds for the normal approximation of smooth functionals in terms of certain Malliavin objects, thereby mimicking the approach taken in \cite{NP-ptrf} on a Gaussian space. In the years to follow, the techniques and results of the two papers \cite{PSTU} and \cite{NPR-ejp} have been generalized and extended e.g.\ to multidimensions and non-smooth probability metrics by various works (see e.g. \cite{Sch16, ET14, PZ1, Zheng15, KRT1, KRT2, KT17}) and, in particular, the Poisson framework has found many fields of relevant applications. We refer to the recent book \cite{PecRei16} for both the theoretical framework and applications of the so-called Malliavin-Stein method on a Poisson space. 
In the seminal paper \cite{Led12}, Ledoux assumed a purely spectral viewpoint in order to derive fourth moment theorems in the framework of functionals of the stationary distribution of some diffusive Markov generator $L$. This approach has then been extended and simplified by the works \cite{ACP} and \cite{CNPP}. Indeed, the spectral viewpoint involving the carr\'{e} du champ operator associated to $L$ was key to proving the fourth moment bound on the Poisson space in \cite{DP17} and 
is also the starting point for our methods in the present article.

Despite the establishment of accurate bounds which have led to both new theoretical insights as well as to new quantitative limit theorems for various models in applications, the question of whether there is a fourth moment theorem also in the discrete Poisson and Rademacher situations has remained open for several years. On the Poisson space indeed, as indicated above, the recent paper \cite{DP17} provided exact, quantitative fourth moment bounds on both the Wasserstein and Kolmogorov 
distances and, in particular, gave a positive answer to this question on the Poisson space. By \textit{exact} we mean that the bounds on the Kolmogorov and Wasserstein distances between the distribution of a normalized multiple Wiener-It\^{o} integral $F$ and the standard normal distribution given in \cite{DP17} are expressed in terms of the fourth cumulant of $F$ only, and hence, no additional term which might account for the discrete nature of general Poisson measures is needed. This fact is even more remarkable in view of de Jong's celebrated CLT for degenerate, non-symmetric $U$-statistics \cite{deJo90} (called homogeneous sums or generalized multilinear forms by de Jong \cite{deJo89, deJo90}) which on top of the fourth moment condition also involves a Lindeberg-Feller type condition, guaranteeing that the maximal influence of each of the independent data random variables on the total variance vanishes asymptotically and which cannot be dispensed with in general.  
In the recent paper \cite{DP16}, the first author and G. Peccati were able to prove a quantitative version of de Jong's result as well as a quantitative extension to multidimensions. This version will be used in Subsection \ref{altproof} 
in order to give an alternative proof of the Wasserstein bound from our main result, Theorem \ref{mt}.

\subsection{Statements of our main results}\label{mainresults}
We now proceed by presenting and discussing our main results. First, we briefly describe the mathematical framework of the paper. For more details and precise definitions we refer to Section \ref{framework} and to the references given there. 
In what follows, we fix a sequence $X=(X_k)_{k\in\N}$ of independent $\{-1,+1\}$-valued random variables on a suitable probability space $(\Omega,\F,\P)$ such that, for 
$k\in\N$, $X_k$ is a Rademacher random variable with success parameter $p_k\in(0,1)$, i.e.
\begin{equation*}
\P(X_k=+1):=p_k\quad\text{and}\quad \P(X_k=-1):=q_k:=1-p_k\,.
\end{equation*}
Furthermore, we denote by $p=(p_k)_{k \in \N}$ and $q=(q_k)_{k \in \N}$ the corresponding sequences of success and failure probabilities.
A sequence $X$ as above is customarily called an \textit{asymmetric, inhomogeneous Rademacher sequence}. We call it \textit{homogeneous} whenever $p_k=p_1$ for all $k\in\N$ and \textit{symmetric} if $p_k=q_k=1/2$ for all $k\in\N$.  
Furthermore, for $m\in\N$, a \textit{symmetric} function $f\in\ell_2(\N^m)$ \textit{vanishing on diagonals}, i.e.\ $f(i_1,\dotsc,i_m)=0$ whenever there are $k\not=l$ in $\{1,\dotsc,m\}$ such that $i_k=i_l$, 
is called a \textit{kernel of order $m$} and the collection of kernels of order $m$ will be denoted by $\ell_0^2(\N)^{\circ m}$. Finally, by $J_m(f)$ we denote the \textit{discrete multiple integral of order $m$} of $f$ with respect to the 
sequence $X$, i.e.\ we have 
\begin{align}\label{intdef}
J_m(f)&:=\sum_{(i_1,\dotsc,i_m)\in\N^m} f(i_1,\dotsc,i_m)Y_{i_1}\cdot\ldots\cdot Y_{i_m}\notag\\
&=m!\sum_{1\leq i_1<\ldots<i_m<\infty}f(i_1,\dotsc,i_m)Y_{i_1}\cdot\ldots\cdot Y_{i_m}\,,
\end{align}
where we denote by $Y=(Y_k)_{k\in\N}$ the normalized sequence corresponding to $X$, given explicitly by 
\begin{equation}\label{defy}
Y_k=\frac{X_k-p_k+q_k}{2\sqrt{p_k q_k}}\,,\quad k\in\N\,.
\end{equation}
Recall that for two real random variables $X$ and $Y$, the \textit{Kolmogorov distance} between their distributions is the supremum norm distance between the corresponding distribution functions, i.e.
\begin{equation*}
 d_\K(X,Y):=\sup_{x\in\R}\babs{\P(X\leq x)-\P(Y\leq x)}\,,
\end{equation*}
and, if $X$ and $Y$ are integrable, then the \textit{Wasserstein distance} between (the distributions of) $X$ and $Y$ is defined as 
\begin{equation*}
 d_\W(X,Y):=\sup_{h\in\Lip(1)}\babs{\E[h(X)]-\E[h(Y)]}\,,
\end{equation*}
where we denote by $\Lip(1)$ the class of all Lipschitz-continuous functions $h:\R\rightarrow\R$ with Lipschitz-constant $1$.
The following theorem and its corollary are the main results of the present paper.

\begin{theorem}[Fourth-moment-influence bound]\label{mt}
 Let $m\in\N$ and let $F=J_m(f)$ be a discrete multiple integral of order m, where $f\in\ell_0^2(\N)^{\circ m}$ is the corresponding kernel 
 such that $\E[F^2]=m!\norm{f}^2_{\ell^2(\N^m)}=1$.  Furthermore, 
 denote by $N\sim N(0,1)$ a standard normal random variable. Then, we have the bound 
 \begin{align}\label{mb}
  d_\W(F,N)&\leq   C_1(m)\sqrt{\babs{\E\bigl[F^4\bigr]-3}}+ C_2(m)\sqrt{\sup_{j\in\N}\Inf_j(f)}\,,
 \end{align}
where the constants $C_1(m)$ and $C_2(m)$ are given by 
\begin{align}\label{constants}
 C_1(m)&=\sqrt{\frac{2}{\pi}}\frac{2m-1}{2m}+\sqrt{\frac{4m-3}{m}}\,,\notag\\
 C_2(m)&=\biggl(\sqrt{\frac{2}{\pi}}\frac{2m-1}{2m} +\sqrt{ \frac{6m-3}{m}}\biggr)\sqrt{\gamma_m}
\end{align}
and $\gamma_m\in(0,\infty)$ is another constant only depending on $m$ (see \eqref{gamma_n} for a possible choice of this constant).\\
Moreover,
\begin{align}\label{KolmogorovTheorem}
d_\K(F,N) &\leq \Big(K_1(m) + K_2(m)((\E[F^4])^{1/4}+1)(\E[F^4])^{1/4}\Big)\sqrt{\absolute{\E[F^4]-3}} \notag\\
&\phantom{{}\leq{}} + \Big(K_3(m) + K_4(m)((\E[F^4])^{1/4}+1)(\E[F^4])^{1/4}\Big)\sqrt{\sup_{j\in\N}\Inf_j(f)}\,,
\end{align}
where the constants $K_1(m), K_2(m), K_3(m)$ and $K_4(m)$ are given by
\begin{align}\label{KolmogorovKonstants}
K_1(m) &= \frac{2m-1+2\sqrt{(8m^2-7)(4m-3)}}{2m}\,, \notag\\
K_2(m) &= \frac{\sqrt{4m^2-3m}}{2m}\,, \notag\\
K_3(m) &= \frac{2m-1+2\sqrt{(8m^2-7)(6m-3)}}{2m}\sqrt{\gamma_m}\,, \notag\\
K_4(m) &= \frac{\sqrt{6m^2-3m}}{2m}\sqrt{\gamma_m}\,.
\end{align}
\end{theorem}

\begin{cor}[Fourth-moment-influence theorem]\label{mtcor}
 Fix an integer $m\geq1$ and, for $n\in\N$, let $F_n=J_m(f_n)$, where $f_n\in\ell_0^2(\N)^{\circ m}$, be a discrete multiple integral of order $m$ such that the following asymptotic properties hold:
 \begin{enumerate}[{\normalfont (i)}]
  \item $\lim_{n\to\infty} \E[F_n^2]=m!\lim_{n\to\infty}\norm{f_n}^2_{\ell^2(\N^m)}=1$.
  \item $\lim_{n\to\infty}\E[F_n^4]=3$.
  \item $\lim_{n\to\infty} \sup_{k\in\N}\Inf_k(f_n)=0$.
 \end{enumerate}
 Then, as $n\to\infty$, $F_n$ converges in distribution to $N$, where $N$ is a standard normal random variable.
\end{cor}

\begin{remark}\label{mtrem}
 \begin{enumerate}[(a)]
  \item Theorem \ref{mt} and Corollary \ref{mtcor} are analogous to the fourth moment bounds/theorems on the Gaussian space (see \cite{NP05} and \cite{NP-ptrf}) and on the Poisson space (see \cite{DP17}). They are also closely connected to 
  de Jong's CLT \cite{deJo90} and its recent quantitative extension \cite{DP16}. Indeed, we will show in Subsection \ref{altproof} how the quantitative version of de Jong's CLT from \cite{DP16} may be applied in order to give 
  an alternative proof of the Wasserstein bound of Theorem \ref{mt} (with slightly different constants). We did not see, however, how to extend this argument to yield a bound on the Kolmogorov distance as well.
  \item Using the hypercontractivity of discrete multiple integrals with respect to a symmetric Rademacher sequence, it is not difficult to see that in the symmetric case and under Condition (i) in Corollary \ref{mtcor},
  the fourth moment condition (ii) is also necessary for the asymptotic normality of $(F_n)_{n\in\N}$. This argument has already been used in \cite{KRT1} in order to find a necessary condition for the asymptotic normality of 
  double integrals in terms of norms of contraction kernels. 
  \item We stress that, in general and in contrast to what has been proved on a Gaussian and on a Poisson space (see \cite{NP05} and \cite{DP17}), the fourth moment condition (ii), however, is not sufficient in order to guarantee asymptotic normality of the sequence $(F_n)_{n\in\N}$. A counterexample for every order $m$ will be given in Example \ref{counter} below and moreover, in Theorem \ref{symtheo}, we show that, in the symmetric case, the fourth moment condition (ii) is sufficient 
  for asymptotic normality if and only if $m=1$.
  \item If $m=1$ and $X$ is a homogeneous Rademacher sequence such that $\E[Y_1^4]\not=3$, then one can do without Condition (iii) in Corollary \ref{mtcor}, i.e.\ in this case an exact fourth moment theorem holds true. This is the content 
  of Corollary \ref{mtcor2}.
  \item It has been known for several years that Condition (iii) above is not necessary in order to have asymptotic normality of $(F_n)_{n\in\N}$. Indeed, let $X$ be symmetric and fix $m\geq2$. Also, for $n\geq m$, we let $F_n$ be given by 
  \begin{align*}
   F_n&=\frac{X_1\cdot\ldots\cdot X_{m-1}}{\sqrt{n-m+1}}\sum_{j=m}^n X_j=J_m(f_n)
   \quad\text{with}\quad\\
   f_n(i_1,\dotsc,i_m)& =\begin{cases}
      \frac{1}{m!\sqrt{n-m+1}}\,,& \text{if }\{i_1,\dotsc,i_m\}=\{1,\dotsc,m-1,l\}\text{ for }m\leq l\leq n\,,\\
      0\,,&\text{otherwise.}
   \end{cases}
    \end{align*}
Then, $X_1\cdot\ldots\cdot X_{m-1}$ is again a symmetric Rademacher random variable (a random sign) which is independent of the sum. Hence, by the classical CLT we conclude that $F_n$ converges in distribution to $N\sim N(0,1)$. 
However, we have $\Inf_1(f_n)=(m!)^{-2}$ for each $n\geq m$. This Example already appears in the monograph \cite[Example 2.1.1]{deJo89} as well as in \cite{KRT1} (for $m=2$) and has also been given in \cite{NPR-aop} in order to show that homogeneous 
polynomial forms in independent Rademacher variables are not universal.
\end{enumerate}
\end{remark}

The next results states that, unless $\E[Y_1^4]=3$, an exact fourth moment theorem holds for integrals of order $m=1$ whenever the Rademacher sequence is homogeneous. 
This, in particular, includes the symmetric case $\P(X_1=1)=\P(X_1=-1)=1/2$. From Example \ref{counter} below it will follow that the restriction $\E[Y_1^4]\not=3$ is necessary.

\begin{cor}\label{mtcor2}
Let $X$ be a homogeneous Rademacher sequence such that $\lambda:=\E[Y_1^4]\not=3$ (which is equivalent to $p_1\not=\frac{1}{2}\pm\frac{1}{2\sqrt{3}}$). Moreover, let $f_n\in\ell_2(\N)$ be a sequence of kernels such that $\lim_{n\to\infty} \norm{f_n}_{\ell^2(\N)}=1$ and $\lim_{n\to\infty}\E[F_n^4]=3$, where $F_n:=J_1(f_n)$, $n\in\N$. Then, as $n\to\infty$, $F_n$ converges in distribution to $N\sim N(0,1)$. 
\end{cor}

\begin{proof}
Fix $f\in\ell^2(\N)$ and consider $F=\sum_{j\in\N}f(j) Y_j$, where we assume that $\sum_{j\in\N} f(j)^2=\Var(F)=1$. Then, 
\begin{equation*}
 F^2=\sum_{j\in\N} f(j)^2 Y_j^2+\sum_{\substack{i,j\in\N:\\i\not=j}}f(i)f(j)Y_iY_j\,,
\end{equation*}
and it is easy to see that these two sums are uncorrelated. Hence, we conclude
\begin{align*}
 \E[F^4]-1&=\Var(F^2)=\sum_{j\in\N}f(j)^4\bigl(\E[Y_j^4]-1\bigr)+2\sum_{\substack{i,j\in\N:\\i\not=j}}f(i)^2f(j)^2\\
 &=(\lambda-1)\sum_{j\in\N} f(j)^4 +2\Bigl(\sum_{j\in\N} f(j)^2\Bigr)^2-2\sum_{j\in\N} f(j)^4\\
 &=(\lambda-3)\sum_{j\in\N} f(j)^4+2\,.
\end{align*}
Hence, 
\begin{equation*}
 \E[F^4]-3=(\lambda-3)\sum_{j\in\N} f(j)^4\,.
\end{equation*}
Now, we have the simple chain of inequalities
\begin{equation*}
\sup_{k\in\N}\Inf_k(f)^2=\sup_{k\in\N} f(k)^4\leq\sum_{j\in\N} f(j)^4\leq \sup_{k\in\N} f(k)^2=\sup_{k\in\N}\Inf_k(f)\,.
\end{equation*}
In particular, since $\lambda\not=3$, we can conclude that 
\begin{align*}
 \sup_{k\in\N}\Inf_k(f)\leq\Bigl(\sum_{j\in\N} f(j)^4\Bigr)^{1/2}= \frac{\sqrt{\babs{\E[F^4]-3}}}{\sqrt{\abs{\lambda-3}}}\,.
\end{align*}
Hence, the result follows from Corollary \ref{mtcor}  by replacing $f$ with the sequence $f_n$, $n\in\N$ and using $\lim_{n\to\infty} \norm{f_n}_{\ell^2(\N)}=1$.
\end{proof}

The following two results demonstrate that, in general even for homogeneous Rademacher sequences, 
there is no exact fourth moment theorem for discrete multiple integrals of order $m\geq2$, i.e.\ that the result in Corollary \ref{mtcor2} is rather exceptional.

\begin{example}[Counterexample to fourth moment condition]\label{counter}
In this example we show that for each fixed integer $m\geq1$ there exist a homogeneous Rademacher sequence $X$ as well as a discrete multiple integral $F$ of order $m$ with $\E[F]=0$, $\Var(F)=1$, $\E[F^4]=3$ such that 
$F$ is not standard normally distributed. By choosing the sequence $F_n:=F$, $n\in\N$, this implies in particular that the fourth moment theorem in general does not hold for Rademacher chaos.
Let an integer $m\geq1$ be given and choose $p_k:=\frac{1}{2} \pm \frac{\sqrt{3^{1/m}-1}}{2\sqrt{3^{1/m}+3}}$ for all $k\in\N$. Since $X_k^2 \equiv 1$ we have  
\begin{align*}
Y_k^2 = 1 + \frac{q_k-p_k}{\sqrt{p_kq_k}}Y_k\,,
\end{align*}
and thus,
\begin{align}\label{y4th}
\E[Y_k^4] = 1 + 2\frac{q_k-p_k}{\sqrt{p_kq_k}} \E[Y_k] + \frac{(q_k-p_k)^2}{p_kq_k} \E[Y_k^2] = 1 + \frac{(q_k-p_k)^2}{p_kq_k}\,,
\end{align}
for every $k \in \N$. By the choice of $p_k$ this makes sure that 
 \begin{align*}
\E[Y_k^4] = 3^{1/m}\,,
\end{align*}
for every $k \in \N$ and, hence, letting $F = Y_1 \cdot \dotsc \cdot Y_m$, we have $F=J_m(f)$, where 
\begin{equation*}
 f(i_1,\dotsc,i_m):=\begin{cases}
                     \frac{1}{m!}\,,&\text{if } \{i_1,\dotsc,i_m\}=\{1,\dotsc,m\}\,,\\
                     0\,,&\text{otherwise,}
                    \end{cases}
\end{equation*}
$\Var(F)=1$ and 
\[\E[F^4]= \E[Y_1^4] \cdot \dotsc \cdot \E[Y_m^4] = 3\,.\]
However, as $F$ obviously only assumes finitely many values, it cannot be normally distributed. 
\end{example}

\begin{theorem}[Counterexample in the symmetric case]\label{symtheo}
Assume that $X=(X_j)_{j\in\N}$ is a symmetric Rademacher sequence. Then, for each $m\geq2$, there is a discrete multiple integral $F$ of order $m$ with respect to $X$ such that $\E[F^2]=1$, $\E[F^4]=3$ which is 
not normally distributed. In particular, the fourth moment theorem fails for chaos of order $m\geq2$.
\end{theorem}

\begin{proof}
First we introduce some notation which helps simplify the presentation of our computations:  
 For integers $1\leq m\leq n$ denote by 
\[\D_m(n):=\{J\subseteq[n]\,:\, \abs{J}=m\}\]
the collection of all $\binom{n}{m}$ $m$-subsets of $[n]:=\{1,\dotsc,n\}$.
We will consider random variables $F$ of the form 
\begin{equation}\label{defw}
 F=\sum_{J\in\D_m(n)}a_J \prod_{i\in J} X_i=\sum_{J\in\D_m(n)}a_J X_J\,, 
\end{equation}
where $a_J\in\R$, $J\in\D_m(n)$, and we write $X_J:=\prod_{i\in J} X_i$. Then, $F$ is a discrete multiple integral of order $m$ such that $\E[F]=0$ and, as in the statement, we assume that 
\[\sum_{J\in\D_m(n)} a_J^2=\E[F^2]=1\,.\]
From the simple fact that, for $I,J,K,L\in\D_m(n)$, we have 
\begin{equation*}
 \E\bigl[X_IX_JX_KX_L\bigr]=\begin{cases}
                             1\,,& I\Delta J= K\Delta L\,,\\
                             0\,,&\text{otherwise}\,,
                            \end{cases}
\end{equation*}
it immediately follows that 
\begin{align}\label{4thform}
 \E[F^4]&=\sum_{\substack{I,J,K,L\in\D_m(n):\\ I\Delta J= K\Delta L}}a_Ia_Ja_Ka_L\,.
\end{align}
It is the simple expression \eqref{4thform} of the fourth moment of $F$ which makes it beneficial for us to use the representation \eqref{defw} of $F$ as indexed by subsets.
Denote by
\[\mathcal{S}_m(n):=\Bigl\{(a_J)_{J\in\D_m(n)}\,:\,\sum_{J\in\D_m(n)} a_J^2=1\Bigr\}\subseteq\R^{\D_m(n)}\]
the sphere of dimension $\binom{n}{m}-1$.
Clearly, the function $g:=g_n:\mathcal{S}_m(n)\rightarrow\R$ given by
\begin{equation*}
 g\bigl(a_J,J\in\D_m(n)\bigr):=\sum_{\substack{I,J,K,L\in\D_m(n):\\ I\Delta J= K\Delta L}}a_Ia_Ja_Ka_L
\end{equation*}
is continuous. 
Let us first consider the case $m=2$ to which the general case will be reduced later on. 
If we can show that, for some $n\in\N$, there are $(b_J)_{J\in\D_2(n)}, (c_J)_{J\in\D_2(n)}\in\mathcal{S}_2(n)$ such that 
\begin{equation*}
 g(b_J,J\in\D_2(n))>3\quad\text{and}\quad g(c_J,J\in\D_2(n))<3\,,
\end{equation*}
then, by the connectedness of $\mathcal{S}_2(n)$, it follows from the intermediate value theorem that there is an $(a_J)_{J\in\D_2(n)}\in\mathcal{S}_2(n)$ such that 
\begin{equation*}
 g(a_J,J\in\D_2(n))=3\,.
\end{equation*}
Then, the variable $F$ defined by \eqref{defw} with $m=2$ and this special sequence $(a_J)_{J\in\D_2(n)}$ will have fourth moment equal to $3$ but it cannot be normally distributed as it assumes only finitely many values.
It thus remains to construct the sequences $(b_J)_{J\in\D_2(n)}, (c_J)_{J\in\D_2(n)}\in\mathcal{S}_2(n)$.
For $n\in\N$, choose $(b_J)_{J\in\D_2(n)}$ such that 
\[b_J:=\frac{1}{\sqrt{\binom{n}{2}}}\,,\quad J\in\D_2(n)\,.\]
In this case we have 
\begin{equation*}
 g_n\bigl((b_J)_{J\in\D_2(n)}\bigr)=\frac{1}{\binom{n}{2}^2}\babs{\{(I,J,K,L)\in\D_2(n)^4\,:\,I\Delta J= K\Delta L\}}\,.
\end{equation*}
By distinguishing the cases $\abs{I\Delta J}=\abs{K\Delta L}=0$, $\abs{I\Delta J}=\abs{K\Delta L}=2$ and $\abs{I\Delta J}=\abs{L\Delta K}=4$ it is not too hard to see that 
\begin{align*}
&\babs{\{(I,J,K,L)\in\D_2(n)^4\,:\,I\Delta J= K\Delta L\}}\\
&=\binom{n}{2}^2+ \binom{n}{2}\cdot (n-2)\cdot 2\cdot (n-2)\cdot2 +\binom{n}{2}\cdot\binom{n-2}{2}\cdot\binom{4}{2} \\
&=\binom{n}{2}^2+4(n-2)^2\binom{n}{2}+ 6 \binom{n}{2}\cdot\binom{n-2}{2} \,.
\end{align*}
Hence, using simple monotonicity arguments, we have 
\begin{align*}
 g_n\bigl((b_J)_{J\in\D_2(n)}\bigr)&=1+\frac{8(n-2)^2}{n(n-1)}+\frac{6(n-2)(n-3)}{n(n-1)}\\
 &\geq1+\frac{8}{3}+1>3
\end{align*}
for all $n\geq4$.
On the other hand, for $n \geq 2$, let $(c_J)_{J\in\D_2(n)}\in\mathcal{S}_2(n)$ be given by 
\begin{equation*}
 c_J:=  \frac{1}{\sqrt{n-1}} 1_J(1)\,,\quad J\in\D_2(n)\,,
 \end{equation*}
such that 
\begin{equation*}
 H:=\sum_{J\in\D_2(n)}c_J X_J= X_1\frac{1}{\sqrt{n-1}}\sum_{k=2}^n X_k=: X_1 S_{n}\,.
\end{equation*}
Then, we have 
\[\E[H^{2r}]=\E[S_n^{2r}]\quad\text{and}\quad \E[H^{2r+1}]=0\]
for all $r\in\N$. In particular, from the computation in the proof of Corollary \ref{mtcor2} with $\lambda=1$ we have
\begin{align*}
 g\bigl((c_J)_{J\in\D_2(n)}\bigr)&=\E[H^4]=\E[S_n^4]=3-2\sum_{k=2}^{n}\frac{1}{(n-1)^2}=3-\frac{2}{n-1}<3
\end{align*}
for all $n\geq2$. By the intermediate value theorem, for $n\geq4$, there hence also exists $(a_J)_{J\in\D_2(n)}\in\mathcal{S}_2(n)$ such that 
\begin{equation*}
 F:=\sum_{J\in\D_2(n)}a_J X_J
\end{equation*}
satisfies 
\[\E[F^4]=3\,,\]
but $F$ cannot be normally distributed. If $m>2$, then letting 
\[G:=X_{n+1}\cdot\ldots\cdot X_{n+m-2}F\]
we have 
\[\E[G^4]=\E[F^4]=3\,.\]
Hence, we have disproved the fourth moment theorem for symmetric Rademacher chaos of every order $m\geq2$.
\end{proof}

\begin{remark}\label{rem4th}
 \begin{enumerate}[(a)]
  \item Theorem \ref{symtheo} and Corollary \ref{mtcor2} give a complete answer about fourth moment theorems in the case of symmetric Rademacher sequences. 
  In particular, Theorem \ref{symtheo} disproves the statement (c)$\Rightarrow$(a) of Proposition 4.6 in \cite{NPR-ejp} dealing with the case $m=2$. 
  \item Example \ref{counter} demonstrates that, also in the non-symmetric case, the fourth moment theorem does not hold in general.  
    \item In the paper \cite{NPPR16} the authors give general conditions for fourth moment theorems of homogeneous multilinear forms in centered i.i.d. random variables $(Y_j)_{j\in\N}$. One of their results (see \cite[Theorem 2.3]{NPPR16}) is that whenever $\E[Y_1^3]=0$ and $\E[Y_1^4]\geq3$, then the fourth moment theorem holds true. By \eqref{y4th} and since $\E[Y_1^3]=\frac{q_1-p_1}{\sqrt{p_1q_1}}$ these two moments conditions are mutually exclusive for homogeneous Rademacher sequences. Hence, the results from \cite{NPPR16} are rather complementary to ours. 
   \end{enumerate}

\end{remark}


\section*{Acknowledgements}
We are grateful to an anonymous referee whose valuable comments and suggestions helped us improve both our results and their presentation. 
We would also like to thank Guangqu Zheng for pointing out a gap in an argument which has now been filled for the final version.

\section{Elements of discrete Malliavin calculus for Rademacher functionals}\label{framework}
In this section we introduce some notation and review several facts about discrete stochastic analysis for Rademacher functionals. Our main reference on this topic is the survey article \cite{Priv08}. However, we 
also refer to the papers \cite{NPR-ejp,KRT1,KRT2} for proofs of certain results. In general, known properties and results are just stated without precisely pointing to a proof.

\subsection{Basic setup and notation}
Recall the definition of an asymmetric, inhomogeneous Rademacher sequence given in Subsection \ref{mainresults}.
Since we are only interested in distributional properties of functionals of the sequence $X$, we may w.l.o.g. assume from the outset that we are working on a canonical space, i.e.\ that 
\begin{equation*}
\Omega=\{-1,+1\}^\N\,,\quad \F=\Pot\bigl(\{-1,+1\}\bigr)^{\otimes\N}\quad\text{and}\quad 
\P=\bigotimes_{k=1}^\infty \Bigl(q_k\delta_{-1}+p_k\delta_{+1}\Bigr)\,,
\end{equation*}
where we denote by $\delta_{\pm1}$ the Dirac measure in $\pm1$. Then, for $k\in\N$, we let $X_k$ be the $k$-th canonical 
projection on $\Omega$, i.e.\ $X_k((\omega_n)_{n\in\N})=\omega_k$.
Recall also the definition \eqref{defy} of the normalized sequence $Y=(Y_k)_{k\in\N}$ corresponding to $X$.
The random variables $Y_k$, $k\in\N$, satisfy the following elementary but important identity
\begin{equation}\label{strucid}
Y_k^2=1+\frac{q_k-p_k}{\sqrt{p_kq_k}}Y_k
\end{equation}
which follows from $X_k^2\equiv1$.
 For $\omega=(\omega_n)_{n\in\N}\in\Omega$ and 
$k\in\N$ we define the sequences 
\[\omega_k^+:=(\omega_1,\dotsc,\omega_{k-1},+1,\omega_{k+1},\dotsc)
\quad\text{and}\quad\omega_k^-:=(\omega_1,\dotsc,\omega_{k-1},-1,\omega_{k+1},\dotsc)\]
and for a functional $F:\Omega\rightarrow\R$ and $k\in\N$ we define $F_k^{\pm}:\Omega\rightarrow\R$ via
\[F_k^+(\omega):=F(\omega_k^+)\quad\text{and}\quad F_k^-(\omega):=F(\omega_k^-)\,.\]
Furthermore, for $F:\Omega\rightarrow\R$ and $k\in\N$ we define 
\[D_kF:=\sqrt{p_k q_k}\bigl(F_k^+-F_k^-\bigr)\quad\text{as well as}\quad DF:=(D_kF)_{k\in\N}:\Omega\times\N\rightarrow\R^{\N}\,.\]
From \cite[Proposition 7.8]{Priv08} we quote the following product rule for the operator $D$:
For all $F,G:\Omega\rightarrow\R$ and $k\in\N$ we have
\begin{align}
D_k(FG)&=FD_kG+GD_kF-\frac{X_k}{\sqrt{p_kq_k}}D_kF D_kG\notag\\
&=FD_kG+GD_kF-2Y_kD_kF D_kG+\frac{q_k-p_k}{\sqrt{p_kq_k}}D_kF D_kG\label{prodD}\,.
\end{align}
Finally, again for $F:\Omega\rightarrow\R$ and $k\in\N$, we introduce the operators 
$D^+F=(D^+_kF)_{k\in\N}$ and $D^+F=(D^+_kF)_{k\in\N}$ via
\begin{equation*}
D^+_kF:=F_k^+-F\quad\text{and}\quad D^-_kF:=F_k^--F\,,\quad k\in\N\,.
\end{equation*}
Note that with this definition we have 
\begin{equation*}
D_kF=\sqrt{p_kq_k}\bigl(D_k^+F-D_k^-F\bigr)\,,\quad k\in\N\,.
\end{equation*}

\subsection{$L^2$-theory and Malliavin operators}
By $\kappa$ we denote from now on the counting measure on $(\N,\Pot(\N))$ and, for $n\in\N$, we write $\kappa^{\otimes n}$ for its $n$-fold product on $(\N^n,\Pot(\N^n))$. Furthermore, we recall the space 
$\ell^2(\N^n)=L^2(\kappa^{\otimes n})$ which consists of all functions $f:\N^n\rightarrow\R$ such that 
\[\sum_{(i_1,\dotsc,i_n)\in\N^n} f^2(i_1,\dotsc,i_n)=\int_{\N^n}f^2 d\kappa^{\otimes n}<\infty\,.\]
By $\ell^2(\N)^{\circ n}$ we denote the subspace of $\ell^2(\N^n)$ consisting of those $f\in\ell^2(\N^n)$ which are symmetric in the sense that $f(i_{\pi(1)},\dotsc,i_{\pi(n)})=f(i_1,\dotsc,i_n)$ for all $(i_1,\dotsc,i_n)\in\N^n$ and all 
permutations $\pi$ of the set $[n]=\{1,\dotsc,n\}$. We write 
\[\Delta_n:=\{(i_1,\dotsc,i_n)\in\N^n\,:\, i_k\not=i_l\text{ for all } k\not=l\}\]
and denote by $\ell_0^2(\N^n)$ the class of all $f\in\ell^2(\N^n)$ such that $f(i_1,\dotsc,i_n)=0$ whenever 
$(i_1,\dotsc,i_n)\in\Delta_n^c:=\N^n\setminus \Delta_n$. Finally, we introduce 
$\ell_0^2(\N)^{\circ n}:= \ell_0^2(\N^n)\cap \ell^2(\N)^{\circ n}$ and call its elements \textit{kernels} in what follows.
If $f:\N^n\rightarrow\R$ is a function, then we denote by $\tilde{f}$ its \textit{canonical symmetrization}, defined via 
\[\tilde{f}(i_1,\dotsc,i_n):=\frac{1}{n!}\sum_{\pi\in\mathbb{S}_n}f(i_{\pi(1)},\dotsc,i_{\pi(n)})\,,\]
where $\mathbb{S}_n$ denotes the group of all permutations of the set $[n]$. Furthermore, for $n\in\N$ and a kernel $f\in\ell_0^2(\N)^{\circ n}$ recall the definition \eqref{intdef} of the \textit{discrete multiple integral} of order $n$ of $f$.
The linear subspace of $L^2(\P)$ consisting of all random variables 
$J_n(f)$, $f\in\ell_0^2(\N)^{\circ n}$, is called the \textit{Walsh chaos} or \textit{Rademacher chaos} of order $n$ and will be denoted by $C_n$ in what follows. An important property of discrete multiple integrals is that they satisfy the isometry relation 
\begin{equation}\label{iso}
\E\bigl[J_m(f)J_n(g)\bigr]=\delta_{n,m} m!\langle f,g\rangle_{\ell^2(\N^m)}\,,
\end{equation}
where $\delta_{n,m}$ denotes \textit{Kronecker's delta symbol}. The fundamental importance of discrete multiple integrals 
is due to the following \textit{chaos decomposition property}: For every $F\in L^2(\P)$ there exists a unique 
sequence of kernels $f_n\in\ell_0^2(\N)^{\circ n}$, $n\in\N_0$, such that $f_0=\E[F]$ and 
\begin{equation}\label{chaosdec}
F=\E[F]+\sum_{n=1}^\infty J_n(f_n)=\sum_{n=0}^\infty J_n(f_n)\,,
\end{equation}
where the series converges in $L^2(\P)$. Note that this, in particular, implies that one has the Hilbert space orthogonal 
decomposition
\begin{equation*}
L^2(\P)=\bigoplus_{n=0}^\infty C_n\,.
\end{equation*}
Denoting by $\proj{\cdot}{n}:L^2(\P)\rightarrow C_n$ the orthogonal projection on $C_n$, by \eqref{chaosdec} we thus have 
\begin{equation}\label{defproj}
\proj{F}{n}=J_n(f_n)\,,\quad n\in\N_0\,,
\end{equation}
whenever $F$ has the chaos decomposition \eqref{chaosdec}.
We denote by $\calS$ the linear subspace of those $F\in L^2(\P)$ whose chaotic decomposition \eqref{chaosdec} is finite, i.e.\ there is an $m\in\N$ (depending on $F$) such that $f_n\equiv0$ for all $n>m$. From \eqref{iso} and the
chaotic decomposition property it is immediate that $\calS$ is dense in $L^2(\P)$.

Let $f\in\ell^2(\N)^{\circ n}$. For $n\in\N$ we define the sub-$\sigma$-field $\F_n:=\sigma(X_1,\dotsc,X_n)=\sigma(Y_1,\dotsc,Y_n)$ of $\F$, and we further let
\begin{equation}\label{martdef}
J_m^{(n)}(f):=\sum_{(i_1,\dotsc,i_m)\in[n]^m}f(i_1,\dotsc,i_m)Y_{i_1}\cdot\dotsc\cdot Y_{i_m}=J_m(f^{(n)})\,,
\end{equation}
where $f^{(n)}(i_1,\dotsc,i_m):=(f\cdot \mathds{1}_{[n]^m})(i_1,\dotsc,i_m)$. Then, it readily follows that $(J_m^{(n)}(f))_{n\in\N}$ is a square-integrable martingale with respect to the filtration $(\F_n)_{n\in\N}$. Moreover, it holds that 
\begin{equation}\label{levy}
 J_m^{(n)}(f)=\E\bigl[J_m(f)\,\bigl|\,\F_n\bigr]\,,\quad n\in\N\,.
\end{equation}
\begin{lemma}\label{l4conv}
The martingale $(J_m^{(n)}(f))_{n\in\N}$ converges $\Prob$-a.s.\ and in $L^4(\Prob)$ to $J_m(f)$. 
In particular, we have
\begin{equation*}
 \lim_{n\to\infty}\E\bigl[J_m^{(n)}(f)^4\bigr]=\E\bigl[J_m(f)^4\bigr]\,.
\end{equation*}
\end{lemma}
\begin{proof}
From \eqref{levy} and martingale theory we obtain that $(J_m^{(n)}(f))_{n\in\N}$ converges almost surely and in $L^1(\P)$ to $J_m(f)$. Furthermore, from \eqref{levy} and the conditional version of Jensen's inequality we conclude that 
\begin{equation*}
 \E\babs{J_m^{(n)}(f)}^4=\E\babs{\E\bigl[J_m(f)\,\bigl|\,\F_n\bigr]}^4\leq\E\Bigl[\E\bigl[\babs{J_m(f)}^4\,\bigl|\,\F_n\bigr]\Bigr]=\E\babs{J_m(f)}^4
\end{equation*}
for each $n\in\N$. Hence, we obtain that 
\begin{equation}\label{l4bounded}
 \sup_{n\in\N}\E\babs{J_m^{(n)}(f)}^4<+\infty
\end{equation}
and the $L^4$-martingale convergence theorem implies that the martingale $(J_m^{(n)}(f))_{n\in\N}$ converges to $J_m(f)$ also in $L^4(\Prob)$.
This proves the lemma.
\end{proof}

In \cite[Proposition 2.1]{KRT2}, the following Stroock type formula for the kernels $f_n$, $n\in\N$, from \eqref{chaosdec} has been given:
\begin{equation}\label{stroock}
 f_n(i_1,\dotsc,i_n)=\frac{1}{n!}\E\bigl[D^n_{i_1,\dotsc,i_n} F\bigr]=\frac{1}{n!}\E\bigl[F\cdot Y_{i_1}\cdot\ldots\cdot Y_{i_n}\bigr]\,,
\end{equation}
where the iterated difference operators $D^n$, $n\in\N_0$, are defined iteratively via $D^0F=F$ and $D^n_{i_1,\dotsc,i_n} F:=D_{i_n}(D^{n-1}_{i_1,\dotsc,i_{n-1}} F)$ for $n\geq1$ and $(i_1,\dotsc,i_n)\in\Delta_n$. Here, $F:\Omega\rightarrow\R$ 
is an arbitrary functional.

By definition, the domain $\dom(D)$ of the \textit{Malliavin derivative operator} is the collection of all $F\in L^2(\P)$ such that the kernels appearing in the chaotic decomposition \eqref{chaosdec} satisfy
\begin{equation*}
 \sum_{n=1}^\infty n n! \norm{f_n}_{\ell^2(\N^n)}^2<\infty\,.
\end{equation*}
It is an important fact that, for $F\in\dom(D)$ with chaotic decomposition \eqref{chaosdec}, we have 
\begin{equation*}
 D_kF=\sum_{n=1}^\infty n J_{n-1}\bigl(f_n(k,\cdot)\bigr)\,,\quad k\in\N\,.
\end{equation*}
Whether $F$ is in $\dom(D)$ or not can also be checked without knowing its chaos decomposition. Indeed, according to Lemma 2.3 from \cite{KRT1} $F\in\dom(D)$ if and only if 
\begin{equation}\label{domdchar}
 \sum_{k=1}^\infty\E\bigl[(D_kF)^2\bigr]=\sum_{k=1}^\infty p_kq_k\E\bigl[(F_k^+-F_k^-)^2\bigr]<\infty\,.
\end{equation}
Note that Lemma 2.3 in \cite{KRT1} actually only deals with the symmetric case $p_k=q_k=1/2$ for all $k\in\N$, but the same proof also works in the general case in view of the general Stroock type formula \eqref{stroock} 
which is fundamental for the proof given in \cite{KRT1}. The next result will be very important in order to apply Stein's method in our framework.

\begin{lemma}\label{Lipdom}
 Suppose that $F\in\dom(D)$ and that $\psi:\R\rightarrow\R$ is Lipschitz-continuous. Then, also $\psi(F)\in\dom(D)$.
\end{lemma}

\begin{proof}
 Let $K\in(0,\infty)$ be a Lipschitz constant for $\psi$. Then, 
 \[\babs{\psi(F)}\leq \babs{\psi(0)}+\babs{\psi(F)-\psi(0)}\leq \babs{\psi(0)}+K\abs{ F}\,.\]
Hence, $\psi(F)\in L^2(\P)$. In order to make sure that $\psi(F)\in\dom(D)$, we are going to verify \eqref{domdchar}. Note that, for $k\in\N$, 
\begin{equation*}
 \babs{D_k\psi(F)}=\sqrt{p_k q_k}\babs{\psi(F_k^+)-\psi(F_k^-)}\leq \sqrt{p_k q_k} K\babs{F_k^+-F_k^-}=K \babs{D_k F}\,.
\end{equation*}
Hence, 
\begin{align*}
  \sum_{k=1}^\infty\E\bigl[\bigl(D_k\psi(F)\bigr)^2\bigr]&\leq K^2  \sum_{k=1}^\infty\E\bigl[(D_kF)^2\bigr]<\infty\,,
\end{align*}
as $F\in\dom(D)$ satisfies \eqref{domdchar}. This proves the lemma.
\end{proof}

The \textit{Ornstein-Uhlenbeck operator} $L$ on $L^2(\P)$ associated with the sequence $X$ is defined by 
\begin{equation}\label{defL}
 LF:=-\sum_{n=1}^\infty n J_{n}(f_n)\,,
\end{equation}
where $F\in L^2(\P)$ is given by \eqref{chaosdec}. Its domain $\dom(L)$ consists precisely of those $F\in L^2(\P)$ whose kernels $f_n$, $n\in\N$, given by \eqref{chaosdec} satisfy
\begin{equation*}
 \sum_{n=1}^\infty n^2 n!\norm{f_n}_{\ell^2(\N^n)}^2<\infty\,.
\end{equation*}
In particular, one has $\calS\subseteq\dom(L)\subseteq\dom(D)$ implying that $L$ is densely defined. Moreover, it is known that $L$ is the infinitesimal generator of a Markovian semigroup, the \textit{Ornstein-Uhlenbeck semigroup} $(P_t)_{t\geq0}$ on 
$L^2(\P)$ defined for $F$ given by \eqref{chaosdec} via 
\[P_tF=\sum_{n=0}^\infty e^{-tn}J_n(f_n)\,.\]
Hence, $-L$ is a closed, positive and self-adjoint operator on $L^2(\P)$. Its spectrum is purely discrete and given by the non-negative integers. Furthermore, from \eqref{defL} 
it follows immediately that $F\in\dom(L)$ is an eigenfunction of $-L$ corresponding to the eigenvalue $n\in\N_0$ if and only if $F\in C_n$. Hence, the projectors given by \eqref{defproj} precisely project on the respective 
eigenspaces of $-L$ and we have $C_n=\ker(L+n\Id)$, $n\in\N_0$, where $\Id$ denotes the identity operator on $L^2(\P)$.

In \cite{Priv08}, the following pathwise representations of the Ornstein-Uhlenbeck operator $L$ are given: 
Whenever $F\in\calS$, we have 
\begin{align}
LF&=-\sum_{k=1}^\infty Y_k D_k F=-\frac12\sum_{k=1}^\infty\bigl(X_k-p_k+q_k\bigr)\bigl(F_k^+-F_k^-\bigr) \label{prepl1}\\
&=\sum_{k=1}^\infty \Bigl(q_k\bigl(F_k^--F\bigr)+p_k\bigl(F_k^+-F\bigr)\Bigr)\notag\\
&=\sum_{k=1}^\infty \Bigl(q_kD_k^-F+p_kD_k^+F\Bigr)\,.\label{prepl2}
\end{align}

In order to provide bounds on the Kolmogorov distance, we also introduce the \textit{divergence} or \textit{Skorohod integral operator} $\delta$ on $L^2(\P\otimes\kappa)$, which is formally defined as the adjoint of $D$, i.e.\ via 
the \textit{integration by parts formula}
\begin{equation}\label{intpartsdelta}
 \E\bigl[F\delta(u)\bigr]=\E\bigl[\langle DF,u\rangle_{\ell^2(\N)}\bigr]=\sum_{k=1}^\infty \E\bigl[(D_k F) u_k\bigr]\,, 
\end{equation}
where $F\in\dom(D)$ and $u=(u_k)_{k\in\N}\in\dom(\delta)$. Note that, for each $k\in\N$, $u_k\in L^2(\P)$ and so there are functions $g_{n+1}:\N^{n+1}\rightarrow\R$, $n\in\N_0$, such that 
$g_{n+1}(k, \cdot)\in\ell_0^2(\N)^{\circ n}$ for each $k\in\N$ and 
\begin{equation*}
 u_k=\sum_{n=0}^\infty J_n\bigl(g_{n+1}(k,\cdot)\bigr)\,,\quad k\in\N\,.
\end{equation*}
Then, it is known that $u\in\dom(\delta)$ if and only if 
\begin{equation*}
 \sum_{n=0}^\infty (n+1)!\norm{\widetilde{g_{n+1}}\mathds{1}_{\Delta_{n+1}}}^2_{\ell^2(\N^{n+1})}<\infty
\end{equation*}
and in this case one has
\begin{equation*}
 \delta(u)=\sum_{n=0}^\infty J_{n+1}\bigl(\widetilde{g_{n+1}}\mathds{1}_{\Delta_{n+1}}\bigr)\,.
\end{equation*}
The three Malliavin operators $D,\delta$ and $L$ are linked in the following way: For $F\in L^2(\P)$ we have $F\in\dom(L)$ if and only if, $F\in\dom(D)$, $DF\in\dom(\delta)$ and, in this case 
\begin{equation}\label{oplink}
 LF=-\delta DF\,.
\end{equation}
In addition, for every $u=(u_k)_{k\in\N}\in\dom(\delta)$, we have the following \textit{Skorohod isometry formula}
\begin{align}\label{Skorohod isometry}
\E[(\delta(u))^2] = \E[\| u \|_{\ell^2(\N)}^2] + \E \Big[ \sum_{k, \ell=1 \atop k \neq \ell}^\infty (D_k u_\ell)(D_\ell u_k) - \sum_{k=1}^\infty (D_ku_k)^2 \Big]. 
\end{align}
Note here that the corresponding Skorohod isometry formula in \cite[Equation (9.5)]{Priv08} contains an error and that the statement \eqref{Skorohod isometry} is a corrected version of it. This has been communicated to us by the author of \cite{Priv08} himself.

As is customary in the theory of infinitesimal generators of Markov semigroups (see \cite{BGL14} for a comprehensive treatment) we define the \textit{carr\'{e} du champ operator} $\Gamma$ associated to $L$ via 
\begin{equation}\label{cdcdef}
 \Gamma(F,G):=\frac12\bigl(L(FG)-FLG-GLF\bigr)\,,
\end{equation}
whenever $F,G\in\dom(L)$ are such that also $FG\in\dom(L)$. 
As $L(FG)$ is centered, and by the self-adjointness of $L$, for such $F,G$, we have the \textit{integration by parts formula} 
\begin{equation}\label{intpartsgamma}
 \E\bigl[\Gamma(F,G)\bigr]=-\E\bigl[F LG\bigr]\,.
\end{equation}

\begin{remark}\label{gammarem}
 In the situation where $L$ is a Markov diffusion generator, one can typically identify a dense algebra $\A\subseteq\dom(L)$ such that $L(\A)\subseteq\A$ and such that $\A$ is closed under sufficiently smooth transformations. Then, one
 usually considers the action of $\Gamma$ on $\A\times\A$ (again, see \cite{BGL14}). Furthermore, in this situation, $\Gamma$ is a derivation in the sense that 
\begin{equation}\label{diffusive}
 \Gamma(\psi(F),G)=\psi'(F)\Gamma(F,G)
\end{equation}
for $\psi$ smooth enough and $F,G\in\A$. Here, however, we are dealing with the non-diffusive Ornstein-Uhlenbeck operator $L$ corresponding to the discrete Rademacher sequence $X$ and, in order to keep track of 
$\Gamma(\psi(F),G)$ for $F,G\in\A:=\calS$ and $\psi$ a continuously differentiable function, we will need a pathwise representation for $\Gamma$ which indeed helps us measure how far $L$ is from being diffusive
in such a way that we can quantify and control the difference between both sides of \eqref{diffusive}. Furthermore, it is not in general true that $\psi(F)\in \calS$ if $F\in\calS$ and $\psi$ is $C^1$. 
This is why we first define an operator $\Gamma_0$ in a pathwise way (see \eqref{gamma0a}), prove a suitable partial integration formula (see Proposition \ref{intparts}) and then show that 
$\Gamma$ and $\Gamma_0$ coincide on $\calS\times\calS$ (see Proposition \ref{gammaprop}). 
\end{remark}

The \textit{pseudo-inverse} $L^{-1}$ of $L$ is defined on the subspace $1^\perp$ of mean zero random variables in $L^2(\P)$ via 
\begin{equation*}
 L^{-1}F:=-\sum_{n=1}^\infty\frac{1}{n} J_n(f_n)\,,
\end{equation*}
where $F$ has chaotic expansion $\sum_{n=1}^\infty J_n(f_n)$. Note that $L^{-1}F\in\dom(L)\subseteq\dom(D)$ for all $F\in 1^\perp$ and that we have 
\begin{align*}
 LL^{-1}F&=F\quad\text{for all } F\in 1^\perp\quad\text{and}\\
 L^{-1}LF&=F-\E[F]\quad\text{for all } F\in\dom(L)\,.
\end{align*}
Using the first of these identities as well as \eqref{intparts} we obtain that,  for $F,G$ such that $G, \, G\, L^{-1}(F-\E(F)) \in \dom L$,
\begin{align}\label{cov}
 \Cov(F,G)&=\E\bigl[G\bigl(F-\E[F]\bigr)\bigr]=\E\bigl[G\cdot LL^{-1}\bigl(F-\E[F]\bigr)\bigr]\notag\\
& =-\E\bigl[\Gamma\bigl(G,L^{-1}\bigl(F-\E[F]\bigr)\bigr]\,.
\end{align}
In particular, if $F=J_m(f)$ is a multiple integral of order $m\in\N$ such that $F^2\in \dom(L)$, then $\E[F]=0$, $L^{-1}F= - m^{-1}F$ and 
\begin{align}\label{varIm}
 \Var(F)=\frac{1}{m}\E\bigl[\Gamma(F,F)\bigr]\,.
\end{align}

\begin{lemma}\label{chaosrank}
Let $m,n\geq1$ be integers and let the discrete multiple integrals $F=J_m(f)$ and $G=J_n(g)$ be in $L^4(\P)$ and given by 
kernels $f\in\ell_0^2(\N)^{\circ m}$ and $g\in\ell_0^2(\N)^{\circ n}$, respectively.
\begin{enumerate}[{\normalfont (a)}]
\item The product $FG\in L^2(\P)$ has a finite chaotic decomposition of the form 
\[FG=\sum_{r=0}^{m+n}\proj{FG}{r}=\sum_{r=0}^{m+n}J_r(h_r)\]
for certain kernels $h_r\in\ell_0^2(\N)^{\circ r}$, $r=0,\dotsc,m+n$.
\item The kernel $h_{m+n}$ in {\normalfont(a)} is explicitly given by $h_{m+n}=f\tilde{\otimes}g \mathds{1}_{\Delta_{m+n}}$,\\ 
where $f\otimes g\in\ell^{2}(\N^{m+n})$ denotes the tensor product of $f$ and $g$ given by 
\[f\otimes g(i_1,\dotsc,i_{m+n})=f(i_1,\dotsc,i_m) g(i_{m+1},\dotsc,i_{m+n})\]
and $f\tilde{\otimes}g$ denotes its canonical symmetrization.
\end{enumerate}
\end{lemma}

The proof of Lemma \ref{chaosrank} is deferred to Section \ref{proofs}.

\begin{remark}\label{crrem}
 Note that the statements (a) and (b) of Lemma \ref{chaosrank} are not direct consequences of the so-called product formula for discrete multiple integrals proved independently in \cite{PrTo} and \cite{Krok15}. Indeed, for these formulas 
 to apply one would have to further assume the square-integrability of the respective involved \textit{contraction kernels} which does not follow from the minimal assumptions of Lemma \ref{chaosrank}. We stress that it is one of the features 
 of the approach via carr\'{e} du champ operators that no precise formulas for the combinatorial coefficients usually appearing in product formulas are needed (see also \cite{Led12}, \cite{ACP} and \cite{DP17}). 
 However, in the case of a symmetric Rademacher sequence Lemma \ref{chaosrank} is a consequence of the product formula for discrete multiple integrals stated as Proposition 2.9 in \cite{NPR-ejp}. 
\end{remark}

\begin{lemma}\label{defgammalemma}
For $F,G\in\dom(D)$, the random functions $(\omega,k)\mapsto D_kF(\omega) D_kG(\omega)$ and 
$(\omega,k)\mapsto \frac{q_k-p_k}{\sqrt{p_kq_k}}Y_k(\omega)D_kF(\omega) D_kG(\omega)$ are in $L^1(\P\otimes\kappa)$. In particular, 
the two series $\sum_{k=1}^\infty D_kF D_kG$ and $\sum_{k=1}^\infty\frac{q_k-p_k}{\sqrt{p_kq_k}}Y_kD_kF D_kG$ are both $\P$-a.s.\ absolutely convergent.
\end{lemma}

\begin{proof}
By the Cauchy-Schwarz inequality for $\kappa$ we have 
\begin{align*}
\sum_{k=1}^\infty \babs{D_kF}\babs{D_kG}&\leq \biggl(\sum_{k=1}^\infty\bigl(D_kF\bigr)^2\biggr)^{1/2}
\biggl(\sum_{k=1}^\infty\bigl(D_kG\bigr)^2\biggr)^{1/2}\,.
\end{align*}
Hence, now using the Cauchy-Schwarz inequality for $\P$ yields
\begin{align}\label{gale1}
\E\Bigl[\sum_{k=1}^\infty \babs{D_kF}\babs{D_kG}\Bigr] \leq\biggl(\E\biggl[\sum_{k=1}^\infty\bigl(D_kF\bigr)^2\biggr]\biggr)^{1/2}
\biggl(\E\biggl[\sum_{k=1}^\infty\bigl(D_kG\bigr)^2\biggr]\biggr)^{1/2} <\infty\,,
\end{align}
as $F,G\in\dom(D)$. Now let us turn to the second series. An easy computation shows that $\E\abs{Y_k}=2\sqrt{p_k q_k}$. 
Hence, using the independence of $Y_k$ and $D_kF D_kG$, $\abs{p_k-q_k}\leq1$ as well as \eqref{gale1} gives
\begin{align}\label{gale2}
\E\Bigl[\sum_{k=1}^\infty \frac{\abs{q_k-p_k}}{\sqrt{p_kq_k}}\babs{Y_k}\babs{D_kF}\babs{D_kG}\Bigr] &=2\sum_{k=1}^\infty\abs{p_k-q_k}\E\abs{D_kF D_kG}\notag\\
&\leq 2\E\Bigl[\sum_{k=1}^\infty \babs{D_kF}\babs{D_kG}\Bigr]<\infty\,.
\end{align}
The $\P$-a.s.\ absolute convergence of both series now follows from \eqref{gale1}, \eqref{gale2} and from the Fubini-Tonelli theorem.
\end{proof}

Thanks to Lemma \ref{defgammalemma}, for $F,G\in\dom(D)$ we can define 
\begin{align}
\Gamma_0(F,G)&:=\sum_{k=1}^\infty \bigl(D_kF\bigr)\bigl( D_kG\bigr)
+\frac12\sum_{k=1}^\infty\frac{q_k-p_k}{\sqrt{p_kq_k}} \bigl(D_kF\bigr)\bigl( D_kG\bigr) Y_k\label{gamma0a}\\
&=\frac12 \sum_{k=1}^\infty \bigl(D_kF\bigr)\bigl( D_kG\bigr)+  \frac12 \sum_{k=1}^\infty \bigl(D_kF\bigr)\bigl( D_kG\bigr)Y_k^2\label{gamma0c} \,,
\end{align}
which is in $L^1(\P)$. Note that \eqref{gamma0c} holds true by virtue of \eqref{strucid}.
In particular, if $p_k=q_k=1/2$ for each $k\in\N$, then 
\begin{equation*}
\Gamma_0(F,G)=\sum_{k=1}^\infty \bigl(D_kF\bigr)\bigl( D_kG\bigr)=\langle DF, DG\rangle_{\ell^2(\N)}\,.
\end{equation*}
By means of a simple computation one immediately checks that for all $k\in\N$
\begin{align*}
\bigl(D_kF\bigr)\bigl( D_kG\bigr)+\frac{q_k-p_k}{2\sqrt{p_kq_k}} \bigl(D_kF\bigr)\bigl( D_kG\bigr) Y_k
=\frac{q_k}{2}\bigl(D_k^-F\bigr)\bigl(D_k^-G\bigr)+\frac{p_k}{2}\bigl(D_k^+F\bigr)\bigl(D_k^+G\bigr)\,.
\end{align*}
Hence, we obtain the following alternative representation for $\Gamma_0$ in terms of the operators $D_k^\pm$ which will be very useful in order to apply Stein's method below.
\begin{equation}\label{gamma0b}
\Gamma_0(F,G)=\frac12\sum_{k=1}^\infty \Bigl(q_k\bigl(D_k^-F\bigr)\bigl(D_k^-G\bigr)+p_k\bigl(D_k^+F\bigr)\bigl(D_k^+G\bigr)\Bigr)
\end{equation}
for all $F,G\in\dom(D)$.

The next result makes sure that $\Gamma_0$ and $\Gamma$ indeed coincide for functionals in $L^4(\P)$ having a finite chaotic decomposition.

\begin{prop}\label{gammaprop}
For all $F,G\in\calS\cap L^4(\P)$ we have $F,G,FG\in\dom(L)$ and $\Gamma(F,G)=\Gamma_0(F,G)$.
\end{prop}

\begin{proof}
Since $F,G\in\calS\cap L^4(\P)$ we have $FG\in\calS$ by Lemma \ref{chaosrank} (a). As $\calS\subseteq\dom(L)\subseteq\dom(D)$ both 
$\Gamma(F,G)$ and $\Gamma_0(F,G)$ are defined. 
Using \eqref{prodD} and \eqref{prepl1} we obtain
\begin{align}\label{gap1}
2\Gamma(F,G)&=L(FG)-FLG-GLF\notag\\
&=-\biggl(\sum_{k=1}^\infty Y_k D_k(FG)-F\sum_{k=1}^\infty Y_kD_kG-G\sum_{k=1}^\infty Y_kD_kF\biggr)\notag\\
&=\sum_{k=1}^\infty\biggl(2Y_k^2+\frac{Y_k(p_k-q_k)}{\sqrt{p_kq_k}}\biggr)D_kF D_kG\notag\\
&=\sum_{k=1}^\infty\biggl(2+2\frac{q_k-p_k}{\sqrt{p_kq_k}}Y_k+\frac{p_k-q_k}{\sqrt{p_kq_k}}Y_k\biggr)D_kF D_kG\notag\\
&=\sum_{k=1}^\infty\biggl(2+\frac{q_k-p_k}{\sqrt{p_kq_k}}Y_k\biggr)D_kF D_kG=2\Gamma_0(F,G)\,.
\end{align}
Here we have used identity \eqref{strucid} to obtain the fourth identity.

\end{proof}

\begin{prop}[Integration by parts]\label{intparts}
Let $H\in\dom(D)$ and $G\in\dom(L)$. Then, we have 
\[\E\bigl[HLG\bigr]=-\E\bigl[\Gamma_0(H,G)\bigr]\,.\]
\end{prop}

\begin{proof}
Let us denote by $H=\sum_{n=0}^\infty J_n(h_n)$ and $G=\sum_{n=0}^\infty J_n(g_n)$ the chaotic decompositions of $H$ and $G$, where $h_n, g_n\in\ell_0^2(\N)^{\circ n}$, $n\in\N_0$, are such that 
\[\sum_{n=1}^\infty n n!\norm{h_n}^2_{\ell^2(\N^n)}<\infty\quad\text{and}\quad \sum_{n=1}^\infty n^2 n!\norm{g_n}^2_{\ell^2(\N^n)}<\infty \,.\]
By \eqref{defL} we have $LG=-\sum_{n=1}^\infty nJ_n(g_n)$. Hence, by virtue of \eqref{iso} we have
 \begin{equation}\label{ip1}
  \E\bigl[HLG\bigr]=-\E\biggl[\biggl(\sum_{m=0}^\infty J_m(h_m)\biggr)\biggl(\sum_{n=1}^\infty nJ_n(g_n)\biggr)\biggr]=-\sum_{n=1}^\infty n n!\langle g_n,h_n\rangle_{\ell^2(\N^n)}\,.
 \end{equation}
On the other hand, using Lemma \ref{defgammalemma} and the fact that $Y_k$ is centered and independent of $D_kH D_kG$ for each $k\in\N$, we obtain that 
\begin{equation}\label{ip2}
 \E\bigl[\Gamma_0(H,G)\bigr]=\E\Bigl[\sum_{k=1}^\infty D_kH D_kG\Bigr]=\sum_{k=1}^\infty \E\bigl[D_kH D_kG\bigr]\,,
\end{equation}
where we could change the order of integration again due to Lemma \ref{defgammalemma}. Now, recall that 
\[D_kH=\sum_{m=1}^\infty  mJ_{m-1}\bigl(h_m(k,\cdot)\bigr)\quad\text{and}\quad D_kG=\sum_{n=1}^\infty nJ_{n-1}\bigl(g_n(k,\cdot)\bigr)\,,\quad k\in\N\,,\]
such that, again by \eqref{iso}, we obtain
\begin{align}\label{ip3}
 \E\bigl[\Gamma_0(H,G)\bigr]&=\sum_{k=1}^\infty\E\bigl[D_kH D_kG\bigr]=\sum_{k=1}^\infty\sum_{m=1}^\infty m^2(m-1)! \langle h_{m}(k,\cdot), g_m(k,\cdot)\rangle_{\ell^2(\N^{m-1})}\notag\\
 &=\sum_{m=1}^\infty m m!\langle g_m,h_m\rangle_{\ell^2(\N^m)}\,.
\end{align}
The result now follows from \eqref{ip1} and \eqref{ip3}\,.
\end{proof}

\section{Useful identities and estimates for multiple integrals}\label{integrals}

The next result is crucial in order to keep track of the non-diffusiveness of the operator $L$ in our bounds.
It is the Rademacher analog of Lemma 2.7 in \cite{DP17} dealing with the corresponding operators on an abstract Poisson space. Its proof is exactly the same as the proof of Lemma 2.7 in \cite{DP17} and is hence omitted.

\begin{lemma}\label{Difflemma}
\begin{enumerate}[{\normalfont (a)}]
 \item For $F:\Omega\rightarrow\R$ and $k\in\N$ we have the identities
 \begin{align}
  D^+_kF^2&=\bigl(D^+_kF\bigr)^2+2F D^+_kF\label{dp2}\,,\\
  D^+_kF^3&=\bigl(D^+_kF\bigr)^3+3F^2 D^+_kF+3F\bigl(D^+_kF\bigr)^2\label{dp3}\,,\\
  D^-_kF^2&=\bigl(D^-_kF\bigr)^2+2F D_k^- F\label{dm2}\,,\\
  D^-_kF^3&=\bigl(D^-_kF\bigr)^3+3F^2D^-_kF+3F\bigl(D^-_k F\bigr)^2\label{dm3}\,.
\end{align}
\item Let $\psi\in C^1(\R)$ be such that $\psi'$ is Lipschitz with minimum Lipschitz-constant $\fnorm{\psi''}$. Then, for $F:\Omega\rightarrow\R$ and $k\in\N$, there are random quantities 
$R_\psi^+(F,k)$ and $R_\psi^-(F,k)$ such that 
\begin{equation*}
 \babs{R_\psi^+(F,k)}\leq\frac{\fnorm{\psi''}}{2}\,,\quad \babs{R_\psi^-(F,k)}\leq\frac{\fnorm{\psi''}}{2}
\end{equation*}
and 
\begin{align*}
 D_k^+\psi(F)&=\psi'(F)D_k^+F +R_\psi^+(F,k)\bigl(D^+_kF\bigr)^2\,,\\
 D_k^-\psi(F)&=\psi'(F)D_k^-F +R_\psi^-(F,k)\bigl(D^-_kF\bigr)^2\,.
\end{align*}
\end{enumerate}
\end{lemma}

\begin{remark}\label{Drem}
 Note that, by virtue of \eqref{dp2} and \eqref{dm2} and by polarization, for $F,G:\Omega\rightarrow\R$ and $k\in\N$ we also deduce the product rules
  \begin{eqnarray}
   D^+_k\bigl(FG\bigr)&=& G D^+_k F+F D^+_k G+ \bigl(D^+_kF \bigr)\bigl(D^+_kG\bigr)\label{mix+}\,,\\
   D^-_k\bigl(FG\bigr)&=& G D_k^- F+F D_k^- G + \bigl(D^-_kF \bigr)\bigl(D_k^-G\bigr)\label{mix-}\,.
  \end{eqnarray}
\end{remark}

\begin{lemma}\label{comblemma}
Let $f\in\ell_0^2(\N)^{\circ m}$, $m \in \N$. Then, we have 
\begin{enumerate}[{\normalfont (a)}]
\item { $\displaystyle (2m)!\norm{f\tilde{\otimes} f}_{\ell^2(\N^{2m})}^2=2\bigl(m!\norm{f}_{\ell^2(\N^{m})}^2\bigr)^2+D_m(f)$\,, where $D_m(f) \in (0,\infty)$ is a constant depending on $f$ and $m$, and}
\item $\displaystyle (2m)!\norm{f\tilde{\otimes} f\mathds{1}_{\Delta_{2m}^c} }_{\ell^2(\N^{2m})}^2\leq \gamma_m m!\norm{f}_{\ell^2(\N^m)}^2 \sup_{j\in\N}\Inf_j(f)$, where
\begin{align}\label{gamma_n}
\gamma_m := 2(2m-1)!\sum_{r=1}^m r!\binom{m}{r}^2 \in (0,\infty)
\end{align}
is a combinatorial constant which only depends on $m$. 
\end{enumerate}
\end{lemma}

\begin{proof}
For a proof of part (a) see e.g.\ identity (5.2.12) in the book \cite{NouPecbook}. Turning to part (b), for every $n,m \in \N$, we use the following abbreviation for tuples of indices: $\bm{i}_n := (i_1, \dotsc, i_n) \in \N^n, \bm{j}_m := (j_1, \dotsc, j_m) \in \N^m$ and $(\bm{i}_n, \bm{j}_m) := (i_1, \dotsc, i_n, j_1, \dotsc, j_m) \in \N^{n+m}$. Then,
\begin{align}
&\norm{(f \tilde{\otimes} f) \1_{\Delta_{2m}^c}}_{\ell^2(\N^{2m})}^2 \leq \norm{(f \otimes f) \1_{\Delta_{2m}^c}}_{\ell^2(\N^{2m})}^2\notag\\
&\qquad = \sum_{(\bm{i}_m, \bm{j}_m) \in \Delta_{2m}^c} f^2(\bm{i}_m) f^2(\bm{j}_m)  = \sum_{(\bm{i}_m, \bm{j}_m) \in \Delta_{2m}^c: \atop \bm{i}_m, \bm{j}_m \in \Delta_m} f^2(\bm{i}_m) f^2(\bm{j}_m) \label{Influence bound proof 1}\,,
\end{align}
where, in the last step, we used the fact that $f$ vanishes on diagonals. We will now count the number of pairs of equal indices in a fixed tuple $(\bm{i}_m, \bm{j}_m) \in \Delta_{2m}^c$ with $\bm{i}_m, \bm{j}_m \in \Delta_m$. Since $\bm{i}_m, \bm{j}_m \in \Delta_m$, each possible pair can only consist of one index taken from the tuple $\bm{i}_m$ and one index taken from tuple $\bm{j}_m$. Thus, each tuple $(\bm{i}_m, \bm{j}_m) \in \Delta_{2m}^c$ with $\bm{i}_m, \bm{j}_m \in \Delta_m$ can contain $r=1, \dotsc, m$ pairs. Now, there are $r! \binom{m}{r}^2$ different ways to build $r$ pairs of two indices in the way described above. By the symmetry of the summands in \eqref{Influence bound proof 1} with respect to the tuples $\bm{i}_m$ and $\bm{j}_m$, respectively, the sum on the right-hand side of \eqref{Influence bound proof 1} can be rewritten in terms of summands containing exactly $r$ pairs of random variables and it follows that
\begin{align}
&\norm{(f \tilde{\otimes} f) \1_{\Delta_{2m}^c}}_{\ell^2(\N^{2m})}^2 \leq \sum_{r=1}^m r! \binom{m}{r}^2 \sum_{(\bm{i}_{m-r}, \bm{j}_{m-r}, \bm{k}_r) \in \Delta_{2m-r}} f^2(\bm{i}_{m-r}, \bm{k}_r) f^2(\bm{j}_{m-r}, \bm{k}_r)\notag\\
&\qquad \leq \sum_{r=1}^m r! \binom{m}{r}^2 \sum_{(\bm{i}_{m-r}, \bm{j}_{m-r}, \bm{k}_r) \in \N^{2m-r}: \atop (\bm{i}_{m-r}, \bm{k}_r), (\bm{j}_{m-r}, \bm{k}_r) \in \Delta_m} f^2(\bm{i}_{m-r}, \bm{k}_r) f^2(\bm{j}_{m-r}, \bm{k}_r) \notag\\
&\qquad \leq \frac{\gamma_m}{2(2m-1)!} \sum_{(\bm{i}_{m-1}, \bm{j}_{m-1}, k) \in \N^{2m-1}: \atop (\bm{i}_{m-1}, k), (\bm{j}_{m-1}, k) \in \Delta_m} f^2(\bm{i}_{m-1}, k) f^2(\bm{j}_{m-1}, k) \label{Influence bound proof 2}\,.
\end{align}
Again, using the fact that $f$ vanishes on diagonals as well as H\"older's inequality it follows from \eqref{Influence bound proof 2} that
\begin{align*}
&\norm{(f \tilde{\otimes} f) \1_{\Delta_{2m}^c}}_{\ell^2(\N^{2m})}^2 \notag\\
&\qquad \leq \frac{\gamma_m}{2(2m-1)!} \sum_{k=1}^\infty \Big( \sum_{\bm{i}_{m-1} \in \Delta_{m-1}} f^2(\bm{i}_{m-1}, k) \Big) \Big( \sum_{\bm{j}_{m-1} \in \Delta_{m-1}} f^2(\bm{j}_{m-1}, k) \Big)\\
&\qquad \leq \frac{\gamma_m}{2(2m-1)!} \Big( \sum_{(\bm{i}_{m-1}, k) \in \Delta_m} f^2(\bm{i}_{m-1}, k) \Big) \sup_{k \in \N} \Big( \sum_{\bm{j}_{m-1} \in \Delta_{m-1}} f^2(\bm{j}_{m-1}, k) \Big)\\
&\qquad = \frac{\gamma_m}{(2m)!} m!\norm{f}_{\ell^2(\N^{m})}^2 \sup_{k \in \N} \Inf_k(f)\,.
\end{align*}
\end{proof}

\begin{lemma}\label{keylemma}
Let $m\in\N$ and suppose that $F=J_m(f)\in C_m$, where $f\in\ell_0^2(\N)^{\circ m}$, is such that 
$\E[F^4]<\infty$. Then, we have 
\begin{align*}
 \sum_{n=1}^{2m-1}\Var\bigl(\proj{F^2}{n}\bigr)\leq \E\bigl[F^4\bigr]-3\bigl(\E[F^2]\bigr)^2+\E[F^2]\;\gamma_m \sup_{j\in\N}\Inf_j(f)\,,
 \end{align*}
where $\gamma_m$ is a finite constant which only depends on $m$ (see \eqref{gamma_n}).
 \end{lemma}

 \begin{proof}
  From Lemma \ref{chaosrank}, we know that $F^2=J_m(f)^2$ has a chaos decomposition of the form 
\begin{equation}\label{kl1}
 F^2=\sum_{n=0}^{2m}\proj{F^2}{n}=\E[F^2]+\sum_{n=1}^{2m-1}\proj{F^2}{n}+ J_{2m}(g_{2m})
\end{equation}
with  $g_{2m}=f\tilde{\otimes}f\mathds{1}_{\Delta_{2m}}$, thus ensuring that $F^2$ is in the domain of $L$. W.l.o.g. we may assume that $\E[F^2]=1$.
From \eqref{kl1} and \eqref{iso} it thus follows that
\begin{align}\label{kl2}
 &\E\bigl[F^4\bigr]-1=\Var\bigl(F^2\bigr)=\sum_{n=1}^{2m}\Var\bigl(\proj{F^2}{n}\bigr)\notag\\
 &\quad =\sum_{n=1}^{2m-1}\Var\bigl(\proj{F^2}{n}\bigr) +(2m)!\norm{f\tilde{\otimes} f \mathds{1}_{\Delta_{2m}}}_{\ell^2(\N^{2m})}^2\notag\\
 &\quad =\sum_{n=1}^{2m-1}\Var\bigl(\proj{F^2}{n}\bigr) +(2m)!\norm{f\tilde{\otimes} f}_{\ell^2(\N^{2m})}^2
 -(2m)!\norm{f\tilde{\otimes} f\mathds{1}_{\Delta_{2m}^c} }_{\ell^2(\N^{2m})}^2   \,.
\end{align}
{ Now, Lemma \ref{comblemma} (a) implies that there is a constant $D_m(f)\in(0,\infty)$ depending on $f$ and m such that}
\begin{equation}\label{vg3}
 (2m!)\norm{f\tilde{\otimes} f}_{\ell^2(\N^{2m})}^2=2(m!)^2\norm{f}_{\ell^2(\N^{m})}^4 +D_m(f)\,.
\end{equation}
Also, 
\begin{equation*}
 2(m!)^2\norm{f}_{\ell^2(\N^{2m})}^4=2\Bigl(\E\bigl[F^2\bigr]\Bigr)^2=2\,.
\end{equation*}
Hence, from \eqref{kl2} and Lemma \ref{comblemma} (b) we see that 
\begin{align*}
 \sum_{n=1}^{2m-1}\Var\bigl(\proj{F^2}{n}\bigr)&\leq \E\bigl[F^4\bigr]-3+ (2m)!\norm{f\tilde{\otimes} f\mathds{1}_{\Delta_{2m}^c}}_{\ell^2(\N^{2m})}^2\\
 &\leq \E\bigl[F^4\bigr]-3+\gamma_m \sup_{j\in\N}\Inf_j(f)\,.
\end{align*}
\end{proof}

\begin{lemma}\label{vargamma}
 Let $m\in\N$ and consider a random variable $F$ such that $F = J_m(f)\in C_m$ and $\E[F^4]<\infty$. Then, $F, F^2\in\dom(L)$ and
\begin{align}
  &\Var\bigl(m^{-1}\Gamma(F,F)\bigr)=\sum_{n=1}^{2m-1} \Bigl(1-\frac{n}{2m}\Bigr)^2 \Var\bigl(\proj{F^2}{n}\bigr)\notag\\
  &\leq \frac{(2m-1)^2}{4m^2}\Bigl(\E\bigl[F^4\bigr]-3\bigl(\E\bigl[F^2\bigr]\bigr)^2 +\E[F^2]\;\gamma_m \sup_{j\in\N}\Inf_j(f) \Bigr)\,. \label{e:cb1}
\end{align}
 Moreover, one also has that
 \begin{align}
  &\frac{1}{m^2} \E[\Gamma(F,F)^2] \leq \E[F^4]\quad\text{and} \label{e:cb2}  \\ 
 &\frac{1}{m} \E[F^2\Gamma(F,F)] \leq \E[F^4]\,. \label{e:cb3}
\end{align}
\end{lemma}

\begin{proof} From  \eqref{kl2} we see that $F^2$ is in the domain of $L$. By homogeneity, without loss of generality we can assume for the rest of the proof that $\E[F^2]=1$. As $LF=-mF$, by the definitions of $\Gamma$ and $L$ we have 
\begin{align}\label{vg1}
 2\Gamma(F,F)&= LF^2-2FLF=\sum_{n=1}^{2m} -n \proj{F^2}{n} +2m \sum_{n=0}^{2m} \proj{F^2}{n}\notag\\
&=\sum_{n=0}^{2m} (2m-n)\proj{F^2}{n}=\sum_{n=0}^{2m-1} (2m-n)\proj{F^2}{n}\,.
\end{align}
By orthogonality, one has that
\begin{align}\label{vg4}
 \Var\bigl(m^{-1}\Gamma(F,F)\bigr)&=\frac{1}{4m^2}\sum_{n=1}^{2m-1}(2m-n)^2\Var\bigl(\proj{F^2}{n}\bigr)\notag\\
 &=\sum_{n=1}^{2m-1}\Bigl(1-\frac{n}{2m}\Bigr)^2\Var\bigl(\proj{F^2}{n}\bigr),
\end{align}
proving the equality in \eqref{e:cb1}. The inequality now follows from 
\begin{equation}\label{coffbound}
 \Bigl(1-\frac{n}{2m}\Bigr)^2\leq \Bigl(1-\frac{1}{2m}\Bigr)^2=\frac{(2m-1)^2}{4m^2}\,,\quad n=1,\dotsc,2m\,,
\end{equation}
as well as from Lemma \ref{keylemma}. Relation \eqref{e:cb2} is an immediate consequences of \eqref{vg4}, \eqref{coffbound} and \eqref{kl2}, and \eqref{e:cb3} follows similarly from \eqref{kl1} and \eqref{vg1} using orthogonality.
\end{proof}

\begin{lemma}\label{remlemma}
 Let $m\in\N$ and let $F=J_m(f)\in L^4(\P)$ be an element of $C_m$. Then, we have 
 \begin{align}\label{remlemmainequality}
  &\frac{1}{2m}\sum_{k=1}^\infty \frac{1}{p_k q_k}\E\babs{D_kF}^4 \notag\\
  &\leq\frac{4m-3}{2m}\Bigl(\E\bigl[F^4\bigr]-3\bigl(\E\bigl[F^2\bigr]\bigr)^2\Bigr)+\frac{6m-3}{2m}\E[F^2]\,\gamma_m \sup_{j\in\N}\Inf_j(f)\,.
 \end{align}
\end{lemma}

\begin{proof}
In order to justify the integration by parts in \eqref{rl1} below we first assume that the stronger integrability condition $F\in L^8(\P)$ holds. Then, by applying Lemma \ref{chaosrank} (a) twice it follows that $F^3\in\mathcal{S}\subseteq\dom(L)$. 
Of course, also $F\in\mathcal{S}\subseteq\dom(L)$ and $F=LL^{-1}F=-m^{-1}LF$. 
Hence, according to Proposition \ref{gammaprop} we can write $\Gamma$ and $\Gamma_0$ interchangeably, and by Proposition \ref{intparts} we have
 \begin{align}\label{rl1}
  \E\bigl[F^4\bigr]&=\E\bigl[F^3 F\bigr]=-\frac{1}{m}\E\bigl[F^3 LF\bigr]=\frac1m \E\bigl[\Gamma_0(F,F^3)\bigr]\,.
 \end{align}
By Lemma \ref{Difflemma} (a) we can write 
\begin{align}\label{rl2}
 \Gamma_0(F,F^3)&=\frac12\sum_{k=1}^\infty\Bigl(q_k D_k^-F D_k^-F^3 +p_k D_k^+F D_k^+F^3\Bigr)\notag\\
 &=\frac12\sum_{k=1}^\infty q_k D_k^-F\Bigl((D_k^-F)^3+3F^2D_k^-F+3F(D_k^-F)^2\Bigr)\notag\\
 &\;+\frac12\sum_{k=1}^\infty p_k D_k^+F\Bigl((D_k^+F)^3+3F^2D_k^+F+3F(D_k^+F)^2\Bigr)\notag\\
 &=\frac12\sum_{k=1}^\infty q_k\Bigl((D_k^-F)^4+3F^2(D_k^-F)^2+3F(D_k^-F)^3\Bigr)\notag\\
 &\;+\frac12\sum_{k=1}^\infty p_k \Bigl((D_k^+F)^4+3F^2(D_k^+F)^2+3F(D_k^+F)^3\Bigr)\,.
\end{align}
Furthermore, 
\begin{align}\label{rl3}
 3F^2\Gamma_0(F,F)&=\frac12\sum_{k=1}^\infty q_k 3F^2(D_k^-F)^2+\frac12\sum_{k=1}^\infty p_k 3F^2(D_k^+F)^2\,.
\end{align}
{ Hence, from \eqref{rl1}, \eqref{rl2} and \eqref{rl3} we obtain
\begin{align}\label{rl4}
 &\frac{3}{m}\E\bigl[F^2\Gamma(F,F)\bigr]-\E\bigl[F^4\bigr]\notag\\
 &\quad =-\frac{1}{2m}\sum_{k=1}^\infty \E\Bigl[q_k\bigl((D_k^-F)^4+3F(D_k^-F)^3\bigr) + p_k\bigl((D_k^+F)^4+3F(D_k^+F)^3\bigr)\Bigr]\,.
\end{align}}
Now, for fixed $k\in\N$, by distinguishing the cases $X_k=+1$ and $X_k=-1$ we obtain 
\begin{align*}
 &q_k\bigl((D_k^-F)^4+3F(D_k^-F)^3\bigr)+p_k\bigl((D_k^+F)^4+3F(D_k^+F)^3\bigr)\\
 &=q_k\bigl((F_k^+-F_k^-)^4-3F_k^+(F_k^+-F_k^-)^3\bigr)\mathds{1}_{\{X_k=+1\}}\\
 &\;+p_k\bigl((F_k^+-F_k^-)^4+3F_k^-(F_k^+-F_k^-)^3\bigr)\mathds{1}_{\{X_k=-1\}}\,.
\end{align*}
Using the fact that $X_k$ is independent of $(F_k^+, F_k^-)$, taking expectations yields
\begin{align}\label{rl5}
 &\E\Bigl[q_k\bigl((D_k^-F)^4+3F(D_k^-F)^3\bigr)+p_k\bigl((D_k^+F)^4+3F(D_k^+F)^3\bigr)\Bigr]\notag\\
 &=2p_kq_k\E\babs{F_k^+-F_k^-}^4-3p_kq_k\E\babs{F_k^+-F_k^-}^4\notag\\
 &=-p_kq_k\E\babs{F_k^+-F_k^-}^4=-\frac{1}{p_k q_k}\E\babs{D_kF}^4\,.
\end{align}
Hence, from \eqref{rl4} and \eqref{rl5} we obtain
\begin{align*}
 \frac{1}{2m}\sum_{k=1}^\infty \frac{1}{p_k q_k}\E\babs{D_kF}^4&=\frac{3}{m}\E\bigl[F^2\Gamma_0(F,F)\bigr]-\E\bigl[F^4\bigr]\\
 &=\frac{3}{m}\E\bigl[F^2\Gamma(F,F)\bigr]-\E\bigl[F^4\bigr]\,.
\end{align*}
Now, using \eqref{kl1}, \eqref{vg1} and orthogonality yields
\begin{align*}
 \frac{3}{m}\E\bigl[F^2\Gamma(F,F)\bigr]-\E\bigl[F^4\bigr]&=3\bigl(\E[F^2]\bigr)^2-\E\bigl[F^4\bigr]\notag\\
 &\;+3\sum_{n=1}^{2m-1}\Bigl(1-\frac{n}{2m}\Bigr)\Var\bigl(\proj{F^2}{n}\bigr)
\end{align*}
and, using Lemma \ref{keylemma}, we obtain
\begin{align}\label{rl6}
 &\frac{3}{m}\E\bigl[F^2\Gamma(F,F)\bigr]-\E\bigl[F^4\bigr]\leq 3\bigl(\E[F^2]\bigr)^2-\E\bigl[F^4\bigr]\notag\\
&\;+ 3 \frac{2m-1}{2m}\Bigl(\E\bigl[F^4\bigr]-3\bigl(\E[F^2]\bigr)^2+\E[F^2]\,\gamma_m \sup_{j\in\N}\Inf_j(f)\Bigr)\notag\\
&\leq\frac{4m-1}{2m}\Bigl(\E\bigl[F^4\bigr]-3\bigl(\E[F^2]\bigr)^2\Bigr)+\frac{6m-3}{2m}\E[F^2]\,\gamma_m \sup_{j\in\N}\Inf_j(f)\,.
\end{align}

Altogether, for $F\in L^8(\P)$,  we have thus proved that 
\begin{align*}
  &\frac{1}{2m}\sum_{k=1}^\infty \frac{1}{p_k q_k}\E\babs{D_kF}^4=\frac{3}{m}\E\bigl[F^2\Gamma(F,F)\bigr]-\E\bigl[F^4\bigr]\\
  &\leq\frac{4m-3}{2m}\Bigl(\E\bigl[F^4\bigr]-3\bigl(\E\bigl[F^2\bigr]\bigr)^2\Bigr)+\frac{6m-3}{2m}\E[F^2]\,\gamma_m \sup_{j\in\N}\Inf_j(f)\,.
 \end{align*}

In the general case that $F=J_m(f)\in L^4(\P)$ we use an approximation argument: For every $n\in\N$, let $F_n:=J_m(f^{(n)})$, where we recall the definition of $f^{(n)}$ from \eqref{martdef}. Note that, for every $n\in\N$ and $p\in[1,\infty)\cup\{+\infty\}$, $F_n\in L^p(\P)$. Thus, \eqref{remlemmainequality} holds for $F_n$, for every $n \in \N$. Now, recall that $F_n=\E[ F \, | \, \F_n]$, for every $n\in\N$. In addition, for every $k,n\in\N$, we have 
\[D_kF_n=mJ_{m-1}(f^{(n)}(k,\cdot))=\E\bigl[D_kF\,\bigl|\,\F_n\bigr]\,.\]
Hence, by Lemma \ref{l4conv} we conclude that, as $n\to\infty$, $F_n\rightarrow F$ and, for every $k\in\N$, $D_kF_n\rightarrow D_kF$ both $\P$-a.s.\ and in $L^4(\P)$. This implies firstly that the right hand side of \eqref{remlemmainequality} for $F_n$ converges to the same quantity for $F$ since, by monotone convergence, we also have $\lim_{n\to\infty}\Inf_j(f^{(n)})=\Inf_j(f)$. On the other hand, by using Fatou's lemma for sums, we obtain
\begin{align*}
 \frac{1}{2m}\sum_{k=1}^\infty \frac{1}{p_k q_k}\E\babs{D_kF}^4&=\frac{1}{2m}\sum_{k=1}^\infty \frac{1}{p_k q_k}\lim_{n\to\infty}\E\babs{D_kF_n}^4\\
& \leq\liminf_{n\to\infty}\frac{1}{2m}\sum_{k=1}^\infty \frac{1}{p_k q_k}\E\babs{D_kF_n}^4\,.
\end{align*}
Therefore, \eqref{remlemmainequality} continues to hold for $F\in L^4(\P)$.

\end{proof}

\begin{lemma}\label{KolmogorovLemma}
Let $m\in\N$ and let $F=J_m(f)\in L^4(\P)$ be an element of $C_m$. Then,
\begin{align*}
0 &\leq \frac{1}{m} \sup_{x \in \R} \E \Big[ \Big\langle \frac{1}{\sqrt{pq}} DF\absolute{DF}, D\1_{\{F>x\}} \Big\rangle_{\ell^2(\N)} \Big] \notag\\
&\qquad \leq \sqrt{(4m-3) \Big( \E[F^4]-3(\Var(F))^2 \Big) + (6m-3) \gamma_m \Var(F) \sup_{j\in\N}\Inf_j(f)} \notag\\
&\phantom{\qquad {}\leq{}} \times \frac{\sqrt{8m^2-7}}{m}\,.
\end{align*}
\end{lemma}

\begin{proof}
The first inequality readily follows from the fact that $(D_kF)(D_k\1_{\{F>x\}})=p_kq_k(F_k^+-F_k^-)(\1_{\{F_k^+>x\}}-\1_{\{F_k^->x\}})\geq0$, for every $k \in \N$. Turning to the second inequality, we want to apply the integration by parts formula from Proposition 2.2 in \cite{KRT2} to further compute the quantity $\E[\langle (pq)^{-1/2} DF\absolute{DF}, D\1_{\{F>x\}} \rangle_{\ell^2(\N)}]$. Therefore, we have to check if the conditions of Proposition 2.2 in \cite{KRT2} are fulfilled for the sequence $u := (u_k)_{k \in \N}$ with $u_k := (p_kq_k)^{-1/2} D_kF\absolute{D_kF}$, for every $k \in \N$. First off, for every $k \in \N$, $(D_k\1_{\{F>x\}})u_k\geq0$, since $(D_kF)(D_k\1_{\{F>x\}})\geq0$. Furthermore, condition (2.14) from \cite{KRT2} can be validated as follows: By the reverse triangle inequality we have $\absolute{D_k\absolute{D_\ell F}} \leq \absolute{D_kD_\ell F}$, for every $k,\ell \in \N$. Hence, by the product formula in \eqref{prodD} and by H\"older's inequality we get, for every $k,\ell \in \N$,
\begin{align*}
&\E[(D_k(D_\ell F \absolute{D_\ell F}))^2] \notag\\
&\qquad = \E \Big[ \Big( (D_\ell F) (D_k\absolute{D_\ell F}) + (D_kD_\ell F)\absolute{D_\ell F} - \frac{X_k}{\sqrt{p_kq_k}}(D_kD_\ell F)(D_k\absolute{D_\ell F}) \Big)^2 \Big] \notag\\
&\qquad \leq \E \Big[ \Big( 2\absolute{D_\ell F}\absolute{D_k D_\ell F} + \frac{1}{\sqrt{p_kq_k}}(D_k D_\ell F)^2 \Big)^2 \Big] \notag\\
&\qquad \leq 8\E[(D_\ell F)^2(D_k D_\ell F)^2] + \frac{2}{p_kq_k}\E[(D_k D_\ell F)^4]\,.
\end{align*}
Thus,
\begin{align}\label{Kolmogorov lemma proof bound 3}
&\E \Big[ \sum_{k,\ell=1}^\infty (D_k u_\ell)^2 \Big] = \E \Big[ \sum_{k,\ell=1}^\infty \Big( D_k \Big( \frac{1}{\sqrt{p_\ell q_\ell}} D_\ell F \absolute{D_\ell F} \Big) \Big)^2 \Big] \notag\\
&\qquad \leq 8\E \Big[ \sum_{\ell=1}^\infty \frac{1}{p_\ell q_\ell}(D_\ell F)^2 \sum_{k=1}^\infty (D_k D_\ell F)^2 \Big]  + 2\E \Big[ \sum_{\ell=1}^\infty \frac{1}{p_\ell q_\ell} \sum_{k=1}^\infty \frac{1}{p_kq_k} (D_k D_\ell F)^4 \Big]\,.
\end{align}
We will now further bound the first summand on the right-hand side of \eqref{Kolmogorov lemma proof bound 3}. For every $k \in \N$, it holds that
\begin{align}
&D_k^+F = (F_k^+ - F_k^-)\1_{\{ X_k = -1 \}} = \frac{1}{\sqrt{p_kq_k}} D_kF \1_{\{ X_k = -1 \}}\,,\label{F_k - F +}\\
&D_k^-F = (F_k^- - F_k^+)\1_{\{ X_k = +1 \}} = -\frac{1}{\sqrt{p_kq_k}} D_kF \1_{\{ X_k = +1 \}}\,.\label{F_k - F -}
\end{align}
Combining \eqref{gamma0b} with \eqref{F_k - F +} and \eqref{F_k - F -} then yields
\begin{align}\label{Kolmogorov lemma proof bound 4}
2\Gamma_0(D_\ell F, D_\ell F) &= \sum_{k=1}^\infty \frac{1}{p_kq_k}(D_k D_\ell F)^2 (q_k\1_{\{ X_k=+1 \}} + p_k\1_{\{ X_k=-1 \}}) \notag\\
&\geq \sum_{k=1}^\infty (D_k D_\ell F)^2\,.
\end{align}
By \eqref{Kolmogorov lemma proof bound 4} and \eqref{e:cb3} we then get
\begin{align}\label{Kolmogorov lemma proof bound 5}
&8\E \Big[ \sum_{\ell=1}^\infty \frac{1}{p_\ell q_\ell}(D_\ell F)^2 \sum_{k=1}^\infty (D_k D_\ell F)^2 \Big] \leq 16\E \Big[ \sum_{\ell=1}^\infty \frac{1}{p_\ell q_\ell}(D_\ell F)^2 \Gamma_0(D_\ell F, D_\ell F) \Big] \notag\\
&\qquad \leq 16(m-1)\E[\norm{(pq)^{-1/4} DF}_{\ell^4(\N)}^4]\,.
\end{align}
Turning to the second summand on the right-hand side of \eqref{Kolmogorov lemma proof bound 3} it follows from the first step in \eqref{Kolmogorov lemma proof bound 4} that
\begin{align}\label{Kolmogorov lemma proof bound 7}
4\E[(\Gamma_0(D_\ell F, D_\ell F))^2] &= \sum_{k=1}^\infty \frac{1}{p_kq_k} \E[(D_k D_\ell F)^4] + \sum_{k,m=1 \atop k \neq m}^\infty \E[(D_kD_\ell F)^2(D_mD_\ell F)^2] \notag\\
&\geq \sum_{k=1}^\infty \frac{1}{p_kq_k} \E[(D_k D_\ell F)^4]\,.
\end{align}
By \eqref{Kolmogorov lemma proof bound 7} and \eqref{e:cb2} we then get
\begin{align}\label{Kolmogorov lemma proof bound 8}
&2\E \Big[ \sum_{\ell=1}^\infty \frac{1}{p_\ell q_\ell} \sum_{k=1}^\infty \frac{1}{p_kq_k} (D_k D_\ell F)^4 \Big] \leq 8\E \Big[ \sum_{\ell=1}^\infty \frac{1}{p_\ell q_\ell}(\Gamma_0(D_\ell F, D_\ell F))^2 \Big] \notag\\
&\qquad \leq 8(m-1)^2\E[\norm{(pq)^{-1/4} DF}_{\ell^4(\N)}^4]\,.
\end{align}
Therefore, combining \eqref{Kolmogorov lemma proof bound 5} and \eqref{Kolmogorov lemma proof bound 8} with \eqref{Kolmogorov lemma proof bound 3} yields
\begin{align}\label{Kolmogorov lemma proof bound 9}
&\E \Big[ \sum_{k,\ell=1}^\infty (D_k u_\ell)^2 \Big] \leq 8(m^2-1)\E[\norm{(pq)^{-1/4} DF}_{\ell^4(\N)}^4]\,.
\end{align}
By virtue of Lemma \ref{remlemma} the quantity on the right-hand side of \eqref{Kolmogorov lemma proof bound 9} is finite. Thus, for every $k \in \N$, $u_k \in \dom(D) \subset L^2(\Omega)$ and admits a chaos representation of the form $u_k = \sum_{n=1}^\infty J_{n-1}(g_n(\, \cdot \, ,k))$ with $g_n \in \ell_0^2(\N)^{\circ n-1} \otimes \ell^2(\N)$, for every $n \in \N$. By the isometry formula in \eqref{iso} it then follows that
\begin{align*}
\sum_{k,\ell=1}^{\infty} \E[(D_k u_\ell)^2] &= \sum_{k,\ell=1}^{\infty} \E \Big[ \Big( \sum_{n=2}^\infty (n-1) J_{n-2}(g_n(\, \cdot \,,k,\ell)) \Big)^2 \Big]\\
&=\sum_{k,\ell=1}^{\infty} \sum_{n=2}^{\infty} (n-1)^2(n-2)! \ellnorm{2}{n-2}{g_n(\, \cdot \,,k,\ell)}^2\\
&=\sum_{n=2}^{\infty} (n-1)(n-1)! \ellnorm{2}{n}{g_n}^2\,.
\end{align*}
So,
\begin{align*}
\sum_{n=2}^{\infty} n! \ellnorm{2}{n}{g_n}^2 \leq \sum_{n=2}^{\infty} 2(n-1)(n-1)! \ellnorm{2}{n}{g_n}^2 = 2\sum_{k,\ell=1}^{\infty} \E[(D_\ell u_k)^2] < \infty
\end{align*}
and $u$ fulfills condition (2.14) from \cite{KRT2}. Note here that condition (2.14) from \cite{KRT2} also implies that $u \in \dom(\delta)$. Now, an application of the integration by parts formula from Proposition 2.2 in \cite{KRT2} yields
\begin{align}\label{Kolmogorov lemma proof bound 1}
&\frac{1}{m} \sup_{x \in \R} \E \Big[ \Big\langle \frac{1}{\sqrt{pq}} DF\absolute{DF}, D\1_{\{F>x\}} \Big\rangle_{\ell^2(\N)} \Big] = \frac{1}{m} \sup_{x \in \R} \E \Big[ \delta\Big( \frac{1}{\sqrt{pq}} DF\absolute{DF} \Big) \1_{\{F>x\}} \Big] \notag\\
&\qquad \leq \frac{1}{m} \E \Big[ \Bigabsolute{\delta\Big( \frac{1}{\sqrt{pq}} DF\absolute{DF} \Big)} \Big] \leq \frac{1}{m} \sqrt{\E \Big[ \Big( \delta\Big( \frac{1}{\sqrt{pq}} DF\absolute{DF} \Big) \Big)^2 \Big]}\,.
\end{align}
The Skorohod isometry formula in \eqref{Skorohod isometry} then yields
\begin{align}\label{Kolmogorov lemma proof bound 2}
&\E \Big[ \Big( \delta\Big( \frac{1}{\sqrt{pq}} DF\absolute{DF} \Big) \Big)^2 \Big] \notag\\
&\qquad \leq \E[\norm{(pq)^{-1/4} DF}_{\ell^4(\N)}^4] + \E \Big[ \sum_{k,\ell=1}^\infty \Big( D_k \Big( \frac{1}{\sqrt{p_\ell q_\ell}} D_\ell F \absolute{D_\ell F} \Big) \Big)^2 \Big]\,.
\end{align}
By plugging \eqref{Kolmogorov lemma proof bound 9} into \eqref{Kolmogorov lemma proof bound 2} we can apply Lemma \ref{remlemma} to deduce that
\begin{align}\label{Kolmogorov lemma proof bound 10}
&\E \Big[ \Big( \delta\Big( \frac{1}{\sqrt{pq}} DF\absolute{DF} \Big) \Big)^2 \Big] \leq (8m^2-7)\E[\norm{(pq)^{-1/4} DF}_{\ell^4(\N)}^4] \notag\\
&\qquad \leq (8m^2-7) \Big( (4m-3) \Big( \E[F^4]-3(\Var(F))^2 \Big) \notag\\
&\phantom{\qquad {}\leq (4m^2-3) \Big(}+ (6m-3) \gamma_m \Var(F) \sup_{j\in\N}\Inf_j(f) \Big)\,.
\end{align}
The proof is now concluded by plugging \eqref{Kolmogorov lemma proof bound 10} into \eqref{Kolmogorov lemma proof bound 1}.
\end{proof}

\section{Proof of Theorem \ref{mt}}\label{mtproof}
First we establish new abstract bounds on the normal approximation of functionals of our Rademacher sequence $X=(X_j)_{j\in\N}$.  

\begin{prop}\label{genbound}
 Let $F\in \dom(D)$ be such that $\E[F]=0$ and let $N\sim N(0,1)$ be a standard normal random variable. 
 Then, we have the bounds
 \begin{align}
  d_\W(F,N)&\leq \sqrt{\frac{2}{\pi}}\E\Babs{1-\Gamma_0\bigl(F,-L^{-1}F\bigr)} +\sum_{k=1}^\infty \frac{1}{\sqrt{p_k q_k}} \E\Bigl[\babs{D_kF}^2\babs{D_kL^{-1}F}\Bigr]   \label{gb1}\\
  &\leq\sqrt{\frac{2}{\pi}}\babs{1-\E[F^2]}+ \sqrt{\frac{2}{\pi}}\sqrt{\Var\bigl(\Gamma_0(F,-L^{-1}F)\bigr)}\notag\\
  &\; +\sum_{k=1}^\infty\frac{1}{\sqrt{p_k q_k}} \E\Bigl[\babs{D_kF}^2\babs{D_kL^{-1}F}\Bigr]  \,. \label{gb2}
 \end{align}
If, furthermore, $F=J_m(f)$ for some $m\in\N$ and some kernel $f\in\ell_0^2(\N)^{\circ m}$ and $\E[F^2]=m!\norm{f}_{\ell^2(\N^m)}^2=1$, then $-L^{-1}F=m^{-1}F$, 
\begin{align*}
 \E\bigl[\Gamma_0(F,-L^{-1}F)\bigr]&=m^{-1}\E\bigl[\Gamma_0(F,F)\bigr]=1\quad\text{and}\\
 \sum_{k=1}^\infty\frac{1}{\sqrt{p_k q_k}} \E\Bigl[\babs{D_kF}^2\babs{D_kL^{-1}F}\Bigr]&=\frac1m\sum_{k=1}^\infty\frac{1}{\sqrt{p_k q_k}} \E\Bigl[\babs{D_kF}^3\Bigr]\\
 &\leq\biggl(\frac1m\sum_{k=1}^\infty { \frac{1}{p_k q_k}} \E\Bigl[\babs{D_kF}^4\Bigr]\biggr)^{1/2}
\end{align*}
so that the previous estimate \eqref{gb2} gives
\begin{align}
  d_\W(F,N)&\leq \sqrt{\frac{2}{\pi}}\sqrt{\Var\bigl(m^{-1}\Gamma_0(F,F)\bigr)}+\biggl(\frac1m\sum_{k=1}^\infty { \frac{1}{p_k q_k}} \E\Bigl[\babs{D_kF}^4\Bigr]\biggr)^{1/2}\,.\label{sb1}
\end{align}
\end{prop}

\begin{proof}
The proof uses Stein's method for normal approximation.  Define the class $\mathscr{F}_\W$ of all continuously differentiable functions $\psi$ on $\R$ such that both $\psi$ and $\psi'$ are Lipschitz-continuous with minimal Lipschitz constants 
\begin{equation}\label{steinsol}
 \fnorm{\psi'}\leq\sqrt{\frac{2}{\pi}}\quad\text{and}\quad\fnorm{\psi''}\leq 2\,.
\end{equation}
 Then, it is well-known (see e.g.\ Theorem 3 of \cite{BPsv} and the references therein) that 
 \begin{equation}\label{pgb1}
  d_\W(F,N)\leq\sup_{\psi\in\mathscr{F}_\W}\babs{\E\bigl[\psi'(F)-F\psi(F)\bigr]}\,.
 \end{equation}
Let us thus fix $\psi\in\mathscr{F}_\W$. By Lemma \ref{Lipdom}, since $\psi$ is Lipschitz, we have $\psi(F)\in\dom(D)$. As $\E[F]=0$, $L^{-1}F$ is well-defined and an element of $\dom(L)$. Hence, as $F=LL^{-1}F$, by Proposition \ref{intparts} we have
\begin{align}\label{pgb2}
 \E\bigl[F\psi(F)\bigr]&=\E\bigl[\psi(F)\cdot LL^{-1}F\bigr]=-\E\bigl[\Gamma_0\bigl(\psi(F),L^{-1}F\bigr)\bigr]\,.
\end{align}
Now, from Equation \eqref{gamma0b} and Lemma \ref{Difflemma} (b) we obtain that 
\begin{align}\label{pgb3}
 &2\Gamma_0\bigl(\psi(F),L^{-1}F\bigr)=\sum_{k=1}^\infty \Bigl(q_k\bigl(D_k^-\psi(F)\bigr)\bigl(D_k^-L^{-1}F\bigr)+p_k\bigl(D_k^+\psi(F)\bigr)\bigl(D_k^+L^{-1}F\bigr)\Bigr)\notag\\
 &=\psi'(F)\sum_{k=1}^\infty q_k\bigl(D_k^-F\bigr)\bigl(D_k^-L^{-1}F\bigr)+\sum_{k=1}^\infty q_k R_\psi^-(F,k)\bigl(D_k^-F\bigr)^2\bigl(D_k^-L^{-1}F\bigr)\notag\\
 &\;+\psi'(F)\sum_{k=1}^\infty p_k\bigl(D_k^+F\bigr)\bigl(D_k^+L^{-1}F\bigr)+\sum_{k=1}^\infty p_k R_\psi^+(F,k)\bigl(D_k^+F\bigr)^2\bigl(D_k^+L^{-1}F\bigr)\notag\\
 &=\psi'(F)\sum_{k=1}^\infty q_k\bigl(D_k^-F\bigr)\bigl(D_k^-L^{-1}F\bigr)+R_+ +\psi'(F)\sum_{k=1}^\infty p_k\bigl(D_k^+F\bigr)\bigl(D_k^+L^{-1}F\bigr)+R_-\notag\\
 &=2\psi'(F)\Gamma_0\bigl(F,L^{-1}F\bigr)+R_+ + R_-\,,
\end{align}
where
\begin{align}\label{pgb4}
 \E\babs{R_+}&\leq\frac{\fnorm{\psi''}}{2}\sum_{k=1}^\infty p_k\E\Babs{(D_k^+F)^2 D_k^+L^{-1}F}\notag\\
 &\leq \sum_{k=1}^\infty p_k\E\Bigl[\Babs{(D_k^+F)^2 D_k^+L^{-1}F} \mathds{1}_{\{X_k=-1\}}\Bigr]\notag\\
 &=\sum_{k=1}^\infty p_kq_k\E\Babs{(F_k^+-F_k^-)^2 \bigl((L^{-1}F)_k^+-(L^{-1}F)_k^-\bigr)}\notag\\
 &=\sum_{k=1}^\infty \frac{1}{\sqrt{p_kq_k}}\E\babs{(D_kF)^2(D_kL^{-1}F)}\,.
\end{align}
Similarly, one shows that 
\begin{align}\label{pgb5}
  \E\babs{R_-}&\leq\sum_{k=1}^\infty \frac{1}{\sqrt{p_kq_k}}\E\babs{(D_kF)^2(D_kL^{-1}F)}\,.
\end{align}
From \eqref{pgb3} we conclude that 
\begin{align*}
 \Babs{\E\bigl[\psi'(F)-F\psi(F)\bigr]}&\leq \Babs{\E\bigl[\psi'(F)\bigr(1-\Gamma_0(F,-L^{-1}F)\bigl)\bigr]}+\frac12\bigl(\E\babs{R_+}+\E\babs{R_-}\bigr)\,,
\end{align*}
which, along with \eqref{pgb1}, \eqref{steinsol}, \eqref{pgb4} and \eqref{pgb5} implies 
\begin{align*}
 d_\W(F,Z)&\leq\sqrt{\frac{2}{\pi}}\E\babs{1-\Gamma_0(F,-L^{-1}F)}+\sum_{k=1}^\infty \frac{1}{\sqrt{p_kq_k}}\E\babs{(D_kF)^2(D_kL^{-1}F)}\,.
\end{align*}
Hence, \eqref{gb1} is proved and \eqref{gb2} now easily follows by first applying the triangle and then the Cauchy-Schwarz inequality. In order to prove \eqref{sb1} we first apply the Cauchy-Schwarz inequality to obtain
\begin{equation}\label{pgb6}
 \sum_{k=1}^\infty \frac{1}{\sqrt{p_kq_k}}\E\babs{D_kF}^3\leq \biggl( \sum_{k=1}^\infty \E\babs{D_kF}^2\biggr)^{1/2}\biggl( \sum_{k=1}^\infty \frac{1}{p_kq_k}\E\babs{D_kF}^4\biggr)^{1/2}\,.
\end{equation}
Now, using $F=J_m(f)$ as well as \eqref{iso} we have 
\begin{align}\label{pgb7}
 \sum_{k=1}^\infty \E\babs{D_kF}^2&=\sum_{k=1}^\infty\E\bigl[\bigl(m J_{m-1}(f(k,\cdot))\bigr)^2\bigr]=m^2\sum_{k=1}^\infty (m-1)!\norm{f(k,\cdot)}^2_{\ell^2(\N^{m-1})}\notag\\
 &=m m!\norm{f}^2_{\ell^2(\N^{m})}=m\E\bigl[F^2\bigr]=m\,.
\end{align}
Hence, from \eqref{pgb6} and \eqref{pgb7} we conclude
\begin{align*}
 \frac1m\sum_{k=1}^\infty \frac{1}{\sqrt{p_kq_k}}\E\babs{D_kF}^3\leq \biggl(\frac1m \sum_{k=1}^\infty \frac{1}{p_kq_k}\E\babs{D_kF}^4\biggr)^{1/2}
\end{align*}
which in turn yields \eqref{sb1}.
\end{proof}

\begin{prop}\label{KolmogorovProposition}
Under the same assumptions as in Proposition \ref{genbound}, one has the bounds
\begin{align}
&d_\K(F,N) \notag\\
&\leq \E[\absolute{1 - \Gamma_0(F,-L^{-1}F)}] \notag\\
&\phantom{{}\leq{}} +\frac{1}{4} \E \Big[ \Big( \absolute{F}+\frac{\sqrt{2\pi}}{4} \Big) \sum_{k=1}^\infty \frac{1}{(pq)^{3/2}} (D_kF)^2 \absolute{-D_kL^{-1}F} (q_k\1_{\{X_k=+1\}}+p_k\1_{\{X_k=-1\}}) \Big] \notag\\
&\phantom{{}\leq{}} +\sup_{x \in \R} \E \Big[ \Big\langle \frac{1}{\sqrt{pq}} (DF) (D\1_{\{F>x\}}), \absolute{-DL^{-1}F} \Big\rangle_{\ell^2(\N)} \Big] \label{KolmogorovBound1}\\
&\leq \E[\absolute{1 - \Var(F)}] + \sqrt{\Var(\Gamma_0(F,-L^{-1}F))} \notag\\
&\phantom{{}\leq{}} + \frac{1}{2\sqrt{2}} \sqrt{\E \Big[ \Big\langle \frac{1}{pq} (DF)^2, (-DL^{-1}F)^2 \Big\rangle_{\ell^2(\N)} \Big]}((\E[F^4])^{1/4} + 1) \notag\\
&\phantom{{}\leq{} + \frac{1}{2\sqrt{2}}} \times \Big( \E \Big[ \Big( \sum_{k=1}^\infty \frac{1}{p_kq_k} (D_kF)^2 (q_k\1_{\{X_k=+1\}}+p_k\1_{\{X_k=-1\}}) \Big)^2 \Big] \Big)^{1/4} \notag\\
&\phantom{{}\leq{}} +\sup_{x \in \R} \E \Big[ \Big\langle \frac{1}{\sqrt{pq}} (DF) (D\1_{\{F>x\}}), \absolute{-DL^{-1}F} \Big\rangle_{\ell^2(\N)} \Big]\,. \label{KolmogorovBound2}
\end{align}
If, furthermore, $F=J_m(f)$ for some $m\in\N$ and some kernel $f\in\ell_0^2(\N)^{\circ m}$ and $\Var(F)=m!\norm{f}_{\ell^2(\N^{m})}^2=1$, then \eqref{KolmogorovBound2} becomes 
\begin{align}
d_\K(F,N) &\leq \frac{1}{m} \sqrt{\Var(\Gamma_0(F,F))} \notag\\
&\phantom{{}\leq{}} + \frac{1}{2\sqrt{2}m} \sqrt{\E[\norm{(pq)^{-1/4} DF}_{\ell^4(\N)}^4]}((\E[F^4])^{1/4} + 1) \notag\\
&\phantom{{}\leq{} + \frac{1}{2\sqrt{2}m}} \times \Big( \E \Big[ \Big( \sum_{k=1}^\infty \frac{1}{p_kq_k} (D_kF)^2 (q_k\1_{\{X_k=+1\}}+p_k\1_{\{X_k=-1\}}) \Big)^2 \Big] \Big)^{1/4} \notag\\
&\phantom{{}\leq{}} + \frac{1}{m} \sup_{x \in \R} \E \Big[ \Big\langle \frac{1}{\sqrt{pq}} DF\absolute{DF}, D\1_{\{F>x\}} \Big\rangle_{\ell^2(\N)} \Big]\,. \label{KolmogorovBound3}
\end{align}
\end{prop}

\begin{proof}
Again, we make use of Stein's method for normal approximation. The starting point is the Stein equation corresponding to the Kolmogorov distance. For $x \in \R$, this equation and its unique bounded solution are given by
\begin{align}\label{Stein equation Kolmogorov}
g'(z) - zg(z) = \1_{(-\infty, x]}(z) - \P(N \leq x)
\end{align}
and
\begin{align}\label{Stein solution Kolmogorov}
g_x(z) := e^{z^2/2} \int_{-\infty}^z (\1_{(\infty,x]}(y) - \P(N \leq x))e^{-y^2/2} \, dy\,,
\end{align}
for every $z \in \R$. Since $g_x$ is not differentiable at the point $x$, one conventionally defines its derivative at the point $x$ by the Stein equation \eqref{Stein equation Kolmogorov} as
\begin{align*}
g_x'(x) := xg_x(x) + 1 - \P(N \leq x)\,.
\end{align*}
This guarantees that \eqref{Stein equation Kolmogorov} really holds in a pointwise sense which is of some importance when dealing with distributions which might have point masses.
It is well known (see e.g.\ Lemma 2.3 in \cite{CheGolSha}) that, for every $x \in \R$, the Stein solution $g_x$ and its derivative can be bounded as follows:
\begin{align}\label{Stein factors Kolmogorov 1}
\absolute{(w+u)g_x(w+u) - (w+v)g_x(w+v)} \leq \Big( \absolute{w} + \frac{\sqrt{2\pi}}{4} \Big) (\absolute{u} + \absolute{v})
\end{align}
and
\begin{align}\label{Stein factors Kolmogorov 2}
\absolute{g_x'(w)} \leq 1\,,
\end{align}
for every $u,v,w \in \R$. Now, by the Stein equation \eqref{Stein equation Kolmogorov} we have, for every $x \in \R$,
\begin{align}\label{Abstract bound proof equation 1}
\P(F \leq x) - \P(N \leq x) = \E[g_x'(F)] - \E[Fg_x(F)]\,.
\end{align}
Note that, for every $x \in \R$, $g_x(F) \in \dom(D)$, since by the mean value theorem and \eqref{Stein factors Kolmogorov 2} we have, for every $k \in \N$,
\begin{align*}
\absolute{D_kg_x(F)} = \sqrt{p_kq_k} \absolute{g_x(F_k^+) - g_x(F_k^-)} \leq \norm{g_x'}_{\infty} \sqrt{p_kq_k} \absolute{F_k^+ - F_k^-} \leq D_kF\,,
\end{align*}
and thus,
\begin{align*}
\E[\norm{Dg_x(F)}_{\ell^2(\N)}^2] = \E \Big[ \sum_{k=1}^\infty (D_kg_x(F))^2 \Big] \leq \E \Big[ \sum_{k=1}^\infty (D_kF)^2 \Big] = \E[\norm{DF}_{\ell^2(\N)}^2] < \infty\,,
\end{align*}
where the last expectation is finite, since $F \in \dom(D)$. Hence, as in the proof of Proposition \ref{genbound}, we can apply the integration by parts formula from Proposition \ref{intparts} for $G=-L^{-1}F$ and $H = g_x(F)$ to \eqref{Abstract bound proof equation 1} and get, for every $x \in \R$,
\begin{align}\label{Abstract bound proof equation 2}
 \P(F \leq x) - \P(N \leq x) &= \E[g_x'(F)] - \E[\Gamma_0(g_x(F), -L^{-1}F)]\,.
\end{align}
Now, for every $x \in \R$, we can write
\begin{align*}
\Gamma_0(g_x(F),-L^{-1}F) &= g_x'(F)\Gamma_0(F,-L^{-1}F) \notag\\
&\phantom{{}={}}+ \frac{1}{2} \sum_{k=1}^\infty \Big(q_k(D_k^-g_x(F)-g_x'(F)D_k^-F)(-D_k^-L^{-1}F) \notag\\
&\phantom{{}={}+ \frac{1}{2} \sum_{k=1}^\infty \Big(}+ p_k(D_k^+g_x(F)-g_x'(F)D_k^+F)(-D_k^+L^{-1}F)\Big)\,.
\end{align*}
Thus, it follows from \eqref{Abstract bound proof equation 2} and \eqref{Stein factors Kolmogorov 2} that
\begin{align}\label{Abstract bound proof equation 3}
\absolute{\P(F \leq x) - \P(N \leq x)} &\leq \E[\absolute{1-\Gamma_0(F,-L^{-1}F)}] \notag\\
&+ \frac{1}{2} \sum_{k=1}^\infty \E \Big[ q_k\absolute{D_k^-g_x(F)-g_x'(F)D_k^-F}\absolute{-D_k^-L^{-1}F} \notag\\
&\phantom{{}+ \frac{1}{2} \sum_{k=1}^\infty \E \Big[}+ p_k\absolute{D_k^+g_x(F)-g_x'(F)D_k^+F}\absolute{-D_k^+L^{-1}F}\Big]\,.
\end{align}
By using \eqref{F_k - F +} and \eqref{F_k - F -} we further deduce that
\begin{align}\label{Abstract bound proof equation 4}
&\frac{1}{2} \sum_{k=1}^\infty \E \Big[ q_k\absolute{D_k^-g_x(F)-g_x'(F)D_k^-F}\absolute{-D_k^-L^{-1}F} \notag\\
&\phantom{\frac{1}{2} \sum_{k=1}^\infty \E \Big[} + p_k\absolute{D_k^+g_x(F)-g_x'(F)D_k^+F}\absolute{-D_k^+L^{-1}F}\Big] \notag\\
&= \frac{1}{2} \sum_{k=1}^\infty \frac{1}{p_kq_k} \E[\absolute{D_kg_x(F)-g_x'(F)D_kF}\absolute{-D_kL^{-1}F} (q_k\1_{\{X_k=+1\}}+p_k\1_{\{X_k=-1\}})]\,.
\end{align}
Therefore, by putting $R_k(F) := D_kg_x(F)-g_x'(F)D_kF$, for every $k \in \N$, and combining \eqref{Abstract bound proof equation 4} with \eqref{Abstract bound proof equation 3} we get
\begin{align}\label{Abstract bound proof equation 5}
&\absolute{\P(F \leq x) - \P(N \leq x)} \notag\\
&\qquad \leq \E[\absolute{1-\Gamma_0(F,-L^{-1}F)}] \notag\\
&\qquad \phantom{{}\leq{}}+ \frac{1}{2} \sum_{k=1}^\infty \frac{1}{p_kq_k} \E[\absolute{R_k(F)} \absolute{-D_kL^{-1}F} (q_k\1_{\{X_k=+1\}}+p_k\1_{\{X_k=-1\}})]\,.
\end{align}
We will now further bound $R_k(F)$ for every $k \in \N$. By the Stein equation \eqref{Stein equation Kolmogorov} we have, for every $k \in \N$,
\begin{align}\label{Chain rule Kolmogorov remainder term}
R_k(F) &= \sqrt{p_kq_k}\int_{D_k^-F}^{D_k^+F} (g_x'(F + t) - g_x'(F)) \, dt \notag\\
&= \sqrt{p_kq_k}\int_{D_k^-F}^{D_k^+F} ((F+t)g_x(F + t) - Fg_x(F)) \, dt \notag\\
&\phantom{{}={}}+ \sqrt{p_kq_k}\int_{D_k^-F}^{D_k^+F} (\1_{\{ F+t \leq x \}} - \1_{\{ F \leq x \}}) \, dt\,.
\end{align}
By virtue of \eqref{Stein factors Kolmogorov 1}, for every $k \in \N$, the first summand on the right-hand side of \eqref{Chain rule Kolmogorov remainder term} can be bounded by
\begin{align}\label{Chain rule Kolmogorov remainder term equation 1}
&\Bigabsolute{\sqrt{p_kq_k}\int_{D_k^-F}^{D_k^+F} ((F+t)g_x(F + t) - Fg_x(F)) \, dt} \notag\\
&\leq \sqrt{p_kq_k} \Big( \absolute{F} + \frac{\sqrt{2\pi}}{4} \Big) \int_{\min \{ D_k^-F, D_k^+F \}}^{\max \{ D_k^-F, D_k^+F \}} \absolute{t} \, dt\,.
\end{align}
Due to \eqref{F_k - F +} and \eqref{F_k - F -}, it follows from \eqref{Chain rule Kolmogorov remainder term equation 1} that, for every $k \in \N$,
\begin{align}\label{Chain rule Kolmogorov remainder term equation 2}
&\Bigabsolute{\sqrt{p_kq_k}\int_{D_k^-F}^{D_k^+F} ((F+t)g_x(F + t) - Fg_x(F)) \, dt} \notag\\
&\leq \sqrt{p_kq_k} \Big( \absolute{F} + \frac{\sqrt{2\pi}}{4} \Big) \int_{\min \{ -D_kF \1_{\{ X_k = +1 \}}, D_kF \1_{\{ X_k = -1 \}} \} /\sqrt{p_kq_k}}^{\max \{ -D_kF \1_{\{ X_k = +1 \}}, D_kF \1_{\{ X_k = -1 \}} \} /\sqrt{p_kq_k}} \absolute{t} \, dt \notag \displaybreak\\
&= \sqrt{p_kq_k} \Big( \absolute{F} + \frac{\sqrt{2\pi}}{4} \Big) \int_0^{\absolute{D_kF} /\sqrt{p_kq_k}} \absolute{t} \, dt \notag\\
&= \frac{1}{2\sqrt{p_kq_k}} (D_kF)^2 \Big( \absolute{F} + \frac{\sqrt{2\pi}}{4} \Big)\,.
\end{align}
To bound the second summand on the right-hand side of \eqref{Chain rule Kolmogorov remainder term}, for every $k \in \N$, we have to separate the following cases
\begin{align}\label{Chain rule Kolmogorov remainder term equation 3}
&\sqrt{p_kq_k}\int_{D_k^-F}^{D_k^+F} (\1_{\{ F+t \leq x \}} - \1_{\{ F \leq x \}}) \, dt \notag\\[5pt]
&= \sqrt{p_kq_k} \Big( \int_{D_k^-F}^{D_k^+F} (\1_{\{ F+t \leq x \}} - \1_{\{ F \leq x \}}) \, dt \1_{\{ X_k=+1, D_kF \geq 0\}} \notag\\
&\phantom{{}={} \sqrt{p_kq_k} \Big(} +\int_{D_k^-F}^{D_k^+F} (\1_{\{ F+t \leq x \}} - \1_{\{ F \leq x \}}) \, dt \1_{\{ X_k=+1, D_kF < 0\}} \notag\\
&\phantom{{}={} \sqrt{p_kq_k} \Big(} +\int_{D_k^-F}^{D_k^+F} (\1_{\{ F+t \leq x \}} - \1_{\{ F \leq x \}}) \, dt \1_{\{ X_k=-1, D_kF \geq 0\}} \notag\\
&\phantom{{}={} \sqrt{p_kq_k} \Big(} +\int_{D_k^-F}^{D_k^+F} (\1_{\{ F+t \leq x \}} - \1_{\{ F \leq x \}}) \, dt \1_{\{ X_k=-1, D_kF < 0\}} \Big)\,.
\end{align}
Now, for every $k \in \N$,
\begin{align*}
&\Bigabsolute{\int_{D_k^-F}^{D_k^+F} (\1_{\{ F+t \leq x \}} - \1_{\{ F \leq x \}}) \, dt} \1_{\{ X_k=+1, D_kF \geq 0\}} \notag\\
&= \Bigabsolute{\int_{-D_kF/\sqrt{p_kq_k}}^{0} (\1_{\{ F_k^+ + t \leq x \}} - \1_{\{ F_k^+ \leq x \}}) \, dt} \1_{\{ X_k=+1, D_kF \geq 0\}} \notag\\
&\leq \frac{1}{\sqrt{p_kq_k}} D_kF \absolute{\1_{\{ F_k^+ - D_kF/\sqrt{p_kq_k} \leq x \}} - \1_{\{ F_k^+ \leq x \}}} \1_{\{ X_k=+1, D_kF \geq 0\}} \notag\\
&= \frac{1}{\sqrt{p_kq_k}} D_kF \absolute{\1_{\{ F_k^- \leq x \}} - \1_{\{ F_k^+ \leq x \}}} \1_{\{ X_k=+1, D_kF \geq 0\}} \notag\\
&= \frac{1}{\sqrt{p_kq_k}} D_kF \absolute{\1_{\{ F_k^+ > x \}} - \1_{\{ F_k^- > x \}}} \1_{\{ X_k=+1, D_kF \geq 0\}} \notag\\
&= \frac{1}{\sqrt{p_kq_k}} D_kF (\1_{\{ F_k^+ > x \}} - \1_{\{ F_k^- > x \}}) \1_{\{ X_k=+1, D_kF \geq 0\}} \notag\\
&= \frac{1}{p_kq_k} (D_kF) (D_k \1_{\{ F > x \}}) \1_{\{ X_k=+1, D_kF \geq 0\}}\,,
\end{align*}
where in the penultimate step we used that, for every $k \in \N$, $F_k^+ \geq F_k^-$ if $D_kF \geq 0$. The remaining quantities in \eqref{Chain rule Kolmogorov remainder term equation 3} can be bounded in a similar way. For every $k \in \N$, we have
\begin{align*}
&\Bigabsolute{\int_{D_k^-F}^{D_k^+F} (\1_{\{ F+t \leq x \}} - \1_{\{ F \leq x \}}) \, dt} \1_{\{ X_k=+1, D_kF < 0\}} \notag\\
&\qquad \leq \frac{1}{p_kq_k} (D_kF) (D_k \1_{\{ F > x \}}) \1_{\{ X_k=+1, D_kF < 0\}}\,,\\
&\Bigabsolute{\int_{D_k^-F}^{D_k^+F} (\1_{\{ F+t \leq x \}} - \1_{\{ F \leq x \}}) \, dt} \1_{\{ X_k=-1, D_kF \geq 0\}} \notag\\
&\qquad \leq \frac{1}{p_kq_k} (D_kF) (D_k \1_{\{ F > x \}}) \1_{\{ X_k=-1, D_kF \geq 0\}}\,,\\
&\Bigabsolute{\int_{D_k^-F}^{D_k^+F} (\1_{\{ F+t \leq x \}} - \1_{\{ F \leq x \}}) \, dt} \1_{\{ X_k=-1, D_kF < 0\}} \notag\\
&\qquad \leq \frac{1}{p_kq_k} (D_kF) (D_k \1_{\{ F > x \}}) \1_{\{ X_k=-1, D_kF < 0\}}\,.
\end{align*}
Thus, it follows from \eqref{Chain rule Kolmogorov remainder term equation 3} that, for every $k \in \N$,
\begin{align}\label{Chain rule Kolmogorov remainder term equation 4}
&\Bigabsolute{\sqrt{p_kq_k}\int_{D_k^-F}^{D_k^+F} (\1_{\{ F+t \leq x \}} - \1_{\{ F \leq x \}}) \, dt} \leq \frac{1}{\sqrt{p_kq_k}} (D_kF) (D_k \1_{\{ F > x \}}).
\end{align}
Combining \eqref{Chain rule Kolmogorov remainder term equation 2} and \eqref{Chain rule Kolmogorov remainder term equation 4} with \eqref{Chain rule Kolmogorov remainder term} yields that, for every $k \in \N$,
\begin{align}\label{Chain rule Kolmogorov remainder term equation 5}
\absolute{R_k(F)} \leq \frac{1}{2\sqrt{p_kq_k}} (D_kF)^2 \Big( \absolute{F} + \frac{\sqrt{2\pi}}{4} \Big) + \frac{1}{\sqrt{p_kq_k}} (D_kF) (D_k \1_{\{ F > x \}}).
\end{align}
The bound \eqref{KolmogorovBound1} now follows by plugging \eqref{Chain rule Kolmogorov remainder term equation 5} into \eqref{Abstract bound proof equation 5} and by the fact that, for every $G \in \dom(D)$ and $k \in \N$, $D_kG$ is independent of $X_k$.\\
The bound \eqref{KolmogorovBound2} is achieved by further bounding the first and second summand on the right-hand side of \eqref{KolmogorovBound1}. For the first summand note that by virtue of Proposition \ref{intparts} we have $\E[\Gamma_0(F,-L^{-1}F)]=\Var(F)$. An application of the triangle and the Cauchy-Schwarz inequality then yields
\begin{align*}
\E[\absolute{1-\Gamma_0(F,-L^{-1}F)}] \leq \E[\absolute{1-\Var(F)}] + \sqrt{\Var(\Gamma_0(F,-L^{-1}F))}.
\end{align*}
For the second summand several applications of the Cauchy-Schwarz inequality as well as an application of the Minkowski inequality lead to the bound
\begin{align*}
&\E \Big[ \Big( \absolute{F}+\frac{\sqrt{2\pi}}{4} \Big) \sum_{k=1}^\infty \frac{1}{(p_kq_k)^{3/2}} (D_kF)^2 \absolute{-D_kL^{-1}F} (q_k\1_{\{X_k=+1\}}+p_k\1_{\{X_k=-1\}}) \Big]\\
&\leq \sqrt{\E \Big[ \sum_{k=1}^\infty \frac{1}{(p_kq_k)^2} (D_kF)^2(-D_kL^{-1}F)^2 (q_k\1_{\{X_k=+1\}}+p_k\1_{\{X_k=-1\}}) \Big]} \notag\\
&\phantom{{} \leq} \times \sqrt{\E \Big[ (\absolute{F}+1)^2 \sum_{k=1}^\infty \frac{1}{p_kq_k} (D_kF)^2 (q_k\1_{\{X_k=+1\}}+p_k\1_{\{X_k=-1\}}) \Big]} \displaybreak\\
&= \sqrt{2\E \Big[ \sum_{k=1}^\infty \frac{1}{p_kq_k} (D_kF)^2(-D_kL^{-1}F)^2 \Big]} \notag\\
&\phantom{{} \leq} \times \sqrt{\E \Big[ (\absolute{F}+1)^2 \sum_{k=1}^\infty \frac{1}{p_kq_k} (D_kF)^2 (q_k\1_{\{X_k=+1\}}+p_k\1_{\{X_k=-1\}}) \Big]}\\
&\leq \sqrt{2\E \Big[ \sum_{k=1}^\infty \frac{1}{p_kq_k} (D_kF)^2(-D_kL^{-1}F)^2 \Big]} (\E[(\absolute{F}+1)^4])^{1/4}\\
&\phantom{{} \leq} \times \Big( \E \Big[ \Big( \sum_{k=1}^\infty \frac{1}{p_kq_k} (D_kF)^2 (q_k\1_{\{X_k=+1\}}+p_k\1_{\{X_k=-1\}}) \Big)^2 \Big] \Big)^{1/4}\\
&\leq \sqrt{2\E \Big[ \Big\langle \frac{1}{pq} (DF)^2, (-DL^{-1}F)^2 \Big\rangle_{\ell^2(\N)} \Big]}((\E[F^4])^{1/4} + 1)\\
&\phantom{{} \leq} \times \Big( \E \Big[ \Big( \sum_{k=1}^\infty \frac{1}{p_kq_k} (D_kF)^2 (q_k\1_{\{X_k=+1\}}+p_k\1_{\{X_k=-1\}}) \Big)^2 \Big] \Big)^{1/4}\,.
\end{align*}
Finally, the bound \eqref{KolmogorovBound3} readily follows from \eqref{KolmogorovBound2} by the fact that, for every $F=J_m(f)$ with $m\in\N$ and $f\in\ell_0^2(\N)^{\circ m}$, it holds that $-L^{-1}F=\frac{1}{m}F$.
\end{proof}

\begin{proof}[End of the proof of Theorem \ref{mt}]
Since $\Gamma_0(F,F)=\Gamma(F,F)$ by Proposition \ref{gammaprop}, the result in \eqref{mb} is an immediate consequence of Bound \eqref{sb1} as well as of Lemma \ref{vargamma} and Lemma \ref{remlemma}.\\
The bound in \eqref{KolmogorovTheorem} follows by applying Lemma \ref{vargamma} to the first, Lemma \ref{remlemma} to the second and Lemma \ref{KolmogorovLemma} to the third summand on the right-hand side of \eqref{KolmogorovBound3}. For the second summand we also use the fact that by virtue of \eqref{e:cb2} we have
\begin{align*}
&\Big( \E \Big[ \Big( \sum_{k=1}^\infty \frac{1}{p_kq_k} (D_kF)^2 (q_k\1_{\{X_k=+1\}}+p_k\1_{\{X_k=-1\}}) \Big)^2 \Big] \Big)^{1/4}\\
&\qquad = (4\E[(\Gamma_0(F,F))^2])^{1/4} \leq \sqrt{2m}(\E[F^4])^{1/4}\,.
\end{align*}
\end{proof}

\subsection{Alternative proof of Theorem \ref{mt} via a quantitative version of de Jong's CLT}\label{altproof}
In this subsection we sketch how one can use the recent quantitative version of de Jong's CLT from \cite{DP16} to give an alternative proof of the Wasserstein bound in Theorem \ref{mt}. 
In order to do this, we briefly review the concepts of Hoeffding decompositions and degenerate $U$-statistics.\\

For $n\in\N$ let $Z_1,\dotsc,Z_n$ be independent random variables on a probability space $(\Omega,\A,\Prob)$ with values in the respective measurable spaces $(E_1,\mathcal{E}_1),\dotsc,(E_n,\mathcal{E}_n)$. 
Furthermore, let $W=\psi(Z_1,\dotsc,Z_n)\in L^1(\Prob)$, where $\psi:\prod_{j=1}^n E_j\rightarrow\R$ is a $\bigotimes_{j=1}^n \mathcal{E}_j$-measurable function. It is a well-known fact that $W$ can be written as 
\begin{equation}\label{hd}
W=\sum_{J\subseteq[n]}W_J\,,
\end{equation}
where the summands $W_J$, $J\subseteq [n]=\{1,\dotsc,n\}$, satisfy the following properties: 
\begin{enumerate}[(i)]
\item For each $J\subseteq [n]$ the random variable $W_J$ is $\G_J$-measurable, where $\G_J:=\sigma(Z_j,j\in J)$.
\item For all $J,K\subseteq[n]$ we have $\E[W_J\,|\,\G_K]=0$ unless $J\subseteq K$.
\end{enumerate}
It is not hard to see that the summands $W_J$, $J\subseteq [n]$, are $\Prob$-a.s.\ uniquely determined by (i) and (ii) and that they are explicitly given by 
\begin{equation*}
W_J=\sum_{K\subseteq J}(-1)^{\abs{J}-\abs{K}}\E\bigl[W\,\bigl|\,\G_K\bigr]\,,\quad J\subseteq [n]\,.
\end{equation*}
The representation \eqref{hd} of $W$ is called the \textit{Hoeffding decomposition} of $W$ and the $W_J$, $J\subseteq[n]$, are called \textit{Hoeffding components}. Moreover, for $1\leq m\leq n$, 
the functional $W$ is called a not necessarily symmetric, (completely) \textit{degenerate $U$-statistic of order $m$}, if the Hoeffding decomposition \eqref{hd} of $W$ is of the form 
\begin{equation}\label{hddeg}
W=\sum_{\substack{J\subseteq[n]:\\ \abs{J}=m}} W_J\,,
\end{equation}
i.e.\ if $W_J=0$ $\Prob$-a.s.\ whenever $\abs{J}\not=m$.\\

The following quantitative extension of a celebrated CLT by de Jong \cite{deJo90}, which is Theorem 1.3 of the recent paper \cite{DP16} by the first author and Peccati, is the essential ingredient for the present proof.

\begin{prop}\label{dejong}
As above, let $W\in L^4(\Prob)$ be a degenerate $U$-statistic of order $1\leq m\leq n$ of the independent random variables $Z_1,\dotsc,Z_n$ such that 
\[\Var(W)=\sum_{\substack{J\subseteq[n]:\\\abs{J}=m}}\E[W_J^2]=1\,. \]
Define 
\begin{equation*}
\rho^2(W):=\max_{1\leq j\leq n}\sum_{\substack{J\subseteq[n]:\\ \abs{J}=m,\, j\in J}} \E[W_J^2]
\end{equation*}
and let $N$ be a standard normal random variable. Then, 
\begin{align*}
d_\W(W,N) &\leq \Bigl(\sqrt{\frac{2}{\pi}}+\frac{4}{3}\Bigr)\sqrt{\babs{\E[W^4]-3}}+\sqrt{\kappa_m}\Bigl(\sqrt{\frac{2}{\pi}}+ \frac{2\sqrt{2}}{\sqrt{3}}\Bigr)\rho(W)\,,
\end{align*}
where $\kappa_m>0$ is a finite constant which only depends on $m$.
\end{prop}

Let $F=J_m(f)$ be as in the statement of Theorem \ref{mt} and recall the definition of $J_m^{(n)}(f)$ from \eqref{martdef}.

\begin{lemma}\label{jpfn}
For each $n\geq m$, the random variable $J_m^{(n)}(f)$ is a (non-symmetric) \textit{degenerate $U$-statistic of order $m$} of the random variables $Y_1,\dotsc,Y_n$.
\end{lemma}

\begin{proof}
Write $W:=J_m^{(n)}(f)$. Using independence, it is easy to see that, for $J=\{i_1,\dotsc,i_m\}$ with 
$1\leq i_1<\dotsc<i_m\leq n$,  the random variables $W_J$ given by 
\begin{equation}\label{hdw}
W_J:=m! f(i_1,\dotsc,i_m) Y_{i_1}\cdot\dotsc \cdot Y_{i_m}
\end{equation}
satisfy (i), (ii) and \eqref{hddeg}.
\end{proof}

For the alternative proof of Theorem \ref{mt} we will also need the following simple lemma.
\begin{lemma}\label{wasslemma}
Let $X,Y,R$ be integrable real-valued random variables on the probability space $(\Omega,\A,\Prob)$ such that 
$\E\abs{R}^2<\infty$. Then, we have 
\begin{equation*}
d_{\W}(X,Y+R)\leq d_{\W}(X,Y)+\E\abs{R}\leq d_{\W}(X,Y)+\sqrt{\E\abs{R}^2}\,.
\end{equation*}
\end{lemma}

\begin{proof}
Let $h\in\Lip(1)$. Then, we have 
\begin{align*}
\babs{\E[h(X)]-\E[h(Y+R)]}&\leq \babs{\E[h(X)]-\E[h(Y)]}+\babs{\E[h(Y)-h(Y+R)]}\\
&\leq d_{\W}(X,Y)+\E\abs{R}\,,
\end{align*}
where we have used that $h$ is $1$-Lipschitz. The result follows by taking the supremum over $h\in\Lip(1)$ and by applying 
the Cauchy-Schwarz inequality.
\end{proof}

\begin{proof}[End of the alternative proof of Theorem \ref{mt}]
Recall $F=J_m(f)$ and, for each $n\geq m$, let  $W_n:=\sigma_n^{-1} J_m^{(n)}(f)$ and $R_n:=F-W_n$, where 
$\sigma_n^2:=\Var(J_m^{(n)}(f))$. From Lemma \ref{wasslemma} we have for $n\geq m$:
\begin{equation}\label{mt1}
d_{\W}(F,N)\leq d_{\W}(W_n,N)+\sqrt{\E\abs{R_n}^2}\,.
\end{equation}
From Lemma \ref{l4conv} we conclude that $\lim_{n\to\infty}\sigma_n^2=\Var(F)=1$ and, furthermore, that 
\begin{align}\label{mt2}
\sqrt{\E\abs{R_n}^2}&\leq\sqrt{\E\Bigl[\bigl(J_m(f)-J_m^{(n)}(f)\bigr)^2\Bigr]}
+\sqrt{\E\babs{J_m^{(n)}(f)}^2}\bigl(1-\sigma_n^{-1}\bigr)\notag\\
&\leq \sqrt{\E\Bigl[\bigl(J_m(f)-J_m^{(n)}(f)\bigr)^2\Bigr]}
+\sqrt{\E\babs{J_m(f)}^2}\bigl(1-\sigma_n^{-1}\bigr)\notag\\
&\longrightarrow0\,, \quad n\to\infty\,.
\end{align}
Moreover, Lemma \ref{jpfn} and Proposition \ref{dejong} imply that 
\begin{align}\label{mt3}
d_{\W}(W_n,N)\leq \Bigl(\sqrt{\frac{2}{\pi}}+\frac{4}{3}\Bigr)\sqrt{\babs{\E[W_n^4]-3}}+\sqrt{\kappa_m}\Bigl(\sqrt{\frac{2}{\pi}}+ \frac{2\sqrt{2}}{\sqrt{3}}\Bigr)\rho(W_n)\,.
\end{align}

Now, Lemma \ref{l4conv} yields that 
\begin{equation}\label{mt4}
\lim_{n\to\infty}\E\bigl[W_n^4\bigr]=\lim_{n\to\infty}\sigma_n^{-4}\E\bigl[J_m^{(n)}(f)^4\bigr]=\E\bigl[F^4\bigr]
\end{equation}
and, recalling the definition of $\rho^2(W_n)$ as well as Lemma \ref{jpfn},
\begin{align}\label{mt5}
\lim_{n\to\infty}\rho^2(W_n)&=\lim_{n\to\infty}\biggl(\sigma_n^{-2}\max_{1\leq j\leq n}
\sum_{\substack{(i_2,\dotsc,i_m)\in (\N\setminus\{j\})^{m-1}:\\ 1\leq i_2<\dotsc<i_m\leq n}}
 (m!)^2 f^2(j,i_2,\dotsc,i_m)\notag\\
&\hspace{4cm}\times\E\bigl[\bigl(Y_j Y_{i_2}\cdot\dotsc\cdot Y_{i_m}\bigr)^2\bigr]\biggr)\notag\\
&=(m!)^2\lim_{n\to\infty}\max_{1\leq j\leq n}
\sum_{\substack{(i_2,\dotsc,i_m)\in (\N\setminus\{j\})^{m-1}:\\ 1\leq i_2<\dotsc<i_m\leq n}}
  f^2(j,i_2,\dotsc,i_m)\notag\\
&=(m!)^2\sup_{j\in\N}\sum_{\substack{(i_2,\dotsc,i_m)\in (\N\setminus\{j\})^{m-1}:\\ 1\leq i_2<\dotsc<i_m<\infty}}
  f^2(j,i_2,\dotsc,i_m)
  =\sup_{j\in\N}\Inf_j(f)\,,
\end{align}
where the next to last identity follows from monotonicity. The result now follows from \eqref{mt1}-\eqref{mt5}.
\end{proof}

\section{Proofs of technical results}\label{proofs}

\begin{proof}[Proof of Lemma \ref{chaosrank}]
We prove (a) and (b) simultaneously. By assumption, $H:=FG\in L^2(\P)$ and, hence, $H$ has a chaotic decomposition of the form 
\begin{equation*}
 H=\E[H]+\sum_{r=1}^\infty J_r(h_r)\,, \quad h_r\in\ell_0^2(\N)^{\circ r}\,.
\end{equation*}
From the second identity in \eqref{stroock} we know that, for $r\in\N$ and for pairwise different $k_1,\dotsc,k_r\in\N$, we have 
\begin{align}\label{cp1}
 h_r(k_1,\dotsc,k_r)&=\frac{1}{r!}\sum_{(i_1,\dotsc,i_m)\in\N^m}f(i_1,\dotsc,i_m)\sum_{(j_1,\dotsc,j_n)\in\N^n}g(j_1,\dotsc,j_n)\notag\\
 &\hspace{3cm} \cdot\E\bigl[Y_{i_1}\cdot\ldots\cdot Y_{i_m}\cdot Y_{j_1}\cdot Y_{j_n}\cdot Y_{k_1}\cdot\ldots\cdot Y_{k_r}\bigr]\,.
\end{align}
Suppose first that $r>m+n$. Then, since $k_1,\dotsc,k_r\in\N$ are pairwise different, for all $(i_1,\dotsc,i_m)\in\N^m$ and $(j_1,\dotsc,j_n)\in\N^n$ we have\\ 
$\{k_1,\dotsc, k_r\}\not\subseteq \{i_1,\dotsc,i_m,j_1,\dotsc,j_n\}$ and thus, by independence, 
\[\E\bigl[Y_{i_1}\cdot\ldots\cdot Y_{i_m}\cdot Y_{j_1}\cdot Y_{j_n}\cdot Y_{k_1}\cdot\ldots\cdot Y_{k_r}\bigr]=0\]
implying $h_r(k_1,\dotsc,k_r)=0$. This proves (a). To prove (b) suppose that $r=m+n$. Then, by the same argument we see that in \eqref{cp1} all summands are equal to zero unless 
$\{k_1,\dotsc, k_r\}=\{i_1,\dotsc,i_m,j_1,\dotsc,j_n\}$. Writing 
\begin{align*}
 \mathcal{M}(k_1,\dotsc,k_{r})&:=\Bigl\{\bigl((i_1,\dotsc,i_m),(j_1,\dotsc,j_n)\bigr)\in\N^m\times \N^n\,:\, \notag\\
 &\hspace{3cm}\{k_1,\dotsc, k_r\}=\{i_1,\dotsc,i_m,j_1,\dotsc,j_n\}\Bigr\}
\end{align*}
we have 
\[\E\bigl[Y_{i_1}\cdot\ldots\cdot Y_{i_m}\cdot Y_{j_1}\cdot Y_{j_n}\cdot Y_{k_1}\cdot\ldots\cdot Y_{k_r}\bigr]=1\]
for all $((i_1,\dotsc,i_m),(j_1,\dotsc,j_n))\in\mathcal{M}(k_1,\dotsc,k_{r})$ and, from \eqref{cp1}, we thus obtain that 
\begin{align*}
 h_{m+n}(k_1,\dotsc,k_{m+n})&=\frac{1}{(m+n)!}\sum_{((i_1,\dotsc,i_m),(j_1,\dotsc,j_n))\in\mathcal{M}(k_1,\dotsc,k_{r})}f(i_1,\dotsc,i_m)g(j_1,\dotsc,j_n)\\
 &=f\tilde{\otimes}g(k_1,\dotsc,k_{m+n})\,,
\end{align*}
proving (b). 
 \end{proof}

\normalem
\bibliography{radebolux}{}

\begin{thebibliography}{CNPP16}

\bibitem[ACP14]{ACP}
E.~Azmoodeh, S.~Campese, and G.~Poly.
\newblock Fourth {M}oment {T}heorems for {M}arkov diffusion generators.
\newblock {\em J. Funct. Anal.}, 266(4):2341--2359, 2014.

\bibitem[BGL14]{BGL14}
D.~Bakry, I.~Gentil, and M.~Ledoux.
\newblock {\em Analysis and geometry of {M}arkov diffusion operators}, volume
  348 of {\em Grundlehren der Mathematischen Wissenschaften [Fundamental
  Principles of Mathematical Sciences]}.
\newblock Springer, Cham, 2014.

\bibitem[BP16]{BPsv}
S.~Bourguin and G.~Peccati.
\newblock {The Malliavin-Stein method on the Poisson space}.
\newblock In G.~Peccati and M.~Reitzner, editors, {\em {Stochastic analysis for
  Poisson point processes}}, Mathematics, Statistics, Finance and Economics,
  chapter~6, pages 185--228. Bocconi University Press and Springer, 2016.

\bibitem[CGS11]{CheGolSha}
L.~H.~Y. Chen, L.~Goldstein, and Q.-M. Shao.
\newblock {\em Normal approximation by {S}tein's method}.
\newblock Probability and its Applications (New York). Springer, Heidelberg,
  2011.

\bibitem[CNPP16]{CNPP}
S.~Campese, I.~Nourdin, G.~Peccati, and G.~Poly.
\newblock Multivariate {G}aussian approximations on {M}arkov chaoses.
\newblock {\em Electron. Commun. Probab.}, 21:paper no. 48, 9 pp., 2016.

\bibitem[dJ89]{deJo89}
P.~de~Jong.
\newblock {\em Central limit theorems for generalized multilinear forms},
  volume~61 of {\em CWI Tract}.
\newblock Stichting Mathematisch Centrum, Centrum voor Wiskunde en Informatica,
  Amsterdam, 1989.

\bibitem[dJ90]{deJo90}
P.~de~Jong.
\newblock A central limit theorem for generalized multilinear forms.
\newblock {\em J. Multivariate Anal.}, 34(2):275--289, 1990.

\bibitem[DP17a]{DP16}
C.~D\"obler and G.~Peccati.
\newblock {Quantiative de Jong theorems in any dimension}.
\newblock {\em Electron. J. Probab.}, 22:paper no. 2, 35 pp., 2017.

\bibitem[DP17b]{DP17}
C.~D\"obler and G.~Peccati.
\newblock {The fourth moment theorem on the Poisson space}.
\newblock {\em to appear in Ann. Probab.}, 2017+.
\newblock {\tt arXiv:1701.03120v2}.

\bibitem[ET14]{ET14}
P.~Eichelsbacher and C.~Th{\"a}le.
\newblock New {B}erry-{E}sseen bounds for non-linear functionals of {P}oisson
  random measures.
\newblock {\em Electron. J. Probab.}, 19:paper no. 102, 25 pp., 2014.

\bibitem[Kro17]{Krok15}
K.~Krokowski.
\newblock {Poisson approximation of Rademacher functionals by the Chen-Stein
  method and Malliavin calculus}.
\newblock {\em Commun. Stoch. Anal.}, 11(2):195--222, 2017.

\bibitem[KRT16]{KRT1}
K.~Krokowski, A.~Reichenbachs, and C.~Th\"ale.
\newblock Berry-{E}sseen bounds and multivariate limit theorems for functionals
  of {R}ademacher sequences.
\newblock {\em Ann. Inst. Henri Poincar\'e Probab. Stat.}, 52(2):763--803,
  2016.

\bibitem[KRT17]{KRT2}
K.~Krokowski, A.~Reichenbachs, and C.~Th\"ale.
\newblock Discrete {M}alliavin--{S}tein method: {B}erry--{E}sseen bounds for
  random graphs and percolation.
\newblock {\em Ann. Probab.}, 45(2):1071--1109, 2017.

\bibitem[KT17]{KT17}
K.~Krokowski and C.~Th\"ale.
\newblock {Multivariate central limit theorems for Rademacher functionals with
  applications}.
\newblock {\em to appear in Electron. J. Probab.}, 2017+.
\newblock {\tt arXiv:1701.07365}.

\bibitem[Led12]{Led12}
M.~Ledoux.
\newblock Chaos of a {M}arkov operator and the fourth moment condition.
\newblock {\em Ann. Probab.}, 40(6):2439--2459, 2012.

\bibitem[MOO10]{MOO10}
E.~Mossel, R.~O'Donnell, and K.~Oleszkiewicz.
\newblock Noise stability of functions with low influences: invariance and
  optimality.
\newblock {\em Ann. of Math. (2)}, 171(1):295--341, 2010.

\bibitem[NP05]{NP05}
D.~Nualart and G.~Peccati.
\newblock Central limit theorems for sequences of multiple stochastic
  integrals.
\newblock {\em Ann. Probab.}, 33(1):177--193, 2005.

\bibitem[NP09]{NP-ptrf}
I~Nourdin and G.~Peccati.
\newblock Stein's method on wiener chaos.
\newblock {\em Probab. Theory Related Fields}, 145(1):75--118, 2009.

\bibitem[NP12]{NouPecbook}
I.~Nourdin and G.~Peccati.
\newblock {\em Normal approximations with {M}alliavin calculus, {F}rom
  {S}tein's method to universality}, volume 192 of {\em Cambridge Tracts in
  Mathematics}.
\newblock Cambridge University Press, Cambridge, 2012.

\bibitem[NPPS16]{NPPR16}
I.~Nourdin, G.~Peccati, G.~Poly, and R.~Simone.
\newblock Classical and free fourth moment theorems: universality and
  thresholds.
\newblock {\em J. Theoret. Probab.}, 29(2):653--680, 2016.

\bibitem[NPR10a]{NPR-aop}
I.~Nourdin, G.~Peccati, and G.~Reinert.
\newblock Invariance principles for homogeneous sums: universality of
  {G}aussian {W}iener chaos.
\newblock {\em Ann. Probab.}, 38(5):1947--1985, 2010.

\bibitem[NPR10b]{NPR-ejp}
I.~Nourdin, G.~Peccati, and G.~Reinert.
\newblock Stein's method and stochastic analysis of {R}ademacher functionals.
\newblock {\em Electron. J. Probab.}, 15:paper no. 55, 1703--1742, 2010.

\bibitem[PR16]{PecRei16}
G.~Peccati and M.~Reitzner.
\newblock {\em {Stochastic Analysis for Poisson Point Processes}}.
\newblock Mathematics, Statistics, Finance and Economics. Bocconi University
  Press and Springer, 2016.

\bibitem[Pri08]{Priv08}
N.~Privault.
\newblock Stochastic analysis of {B}ernoulli processes.
\newblock {\em Probab. Surv.}, 5:435--483, 2008.

\bibitem[PSTU10]{PSTU}
G.~Peccati, J.~L. Sol{{\'e}}, M.~S. Taqqu, and F.~Utzet.
\newblock Stein's method and normal approximation of {P}oisson functionals.
\newblock {\em Ann. Probab.}, 38(2):443--478, 2010.

\bibitem[PT05]{PecTu05}
G.~Peccati and C.~A. Tudor.
\newblock Gaussian limits for vector-valued multiple stochastic integrals.
\newblock In {\em S\'eminaire de {P}robabilit\'es {XXXVIII}}, volume 1857 of
  {\em Lecture Notes in Math.}, pages 247--262. Springer, Berlin, 2005.

\bibitem[PT15]{PrTo}
N.~Privault and G.L. Torrisi.
\newblock The {S}tein and {C}hen-{S}tein methods for functionals of
  non-symmetric {B}ernoulli processes.
\newblock {\em ALEA Lat. Am. J. Probab. Math. Stat.}, 12(1):309--356, 2015.

\bibitem[PZ10]{PZ1}
G.~Peccati and C.~Zheng.
\newblock Multi-dimensional {G}aussian fluctuations on the {P}oisson space.
\newblock {\em Electron. J. Probab.}, 15:paper no. 48, 1487--1527, 2010.

\bibitem[Sch16]{Sch16}
M.~Schulte.
\newblock Normal approximation of {P}oisson functionals in {K}olmogorov
  distance.
\newblock {\em J. Theoret. Probab.}, 29(1):96--117, 2016.

\bibitem[Zhe17]{Zheng15}
G.~Zheng.
\newblock Normal approximation and almost sure central limit theorem for
  non-symmetric {R}ademacher functionals.
\newblock {\em Stochastic Process. Appl.}, 127(5):1622--1636, 2017.

\end{thebibliography}
\bibliographystyle{alpha}
\end{document}